\RequirePackage{fix-cm}
\documentclass{svjour3}       
\smartqed  

\usepackage{graphicx}%
\usepackage{multirow}%
\usepackage{amsmath,amssymb,amsfonts}%
\usepackage{mathrsfs}%
\usepackage[title]{appendix}%
\usepackage{xcolor}%
\usepackage{textcomp}%
\usepackage{enumitem}
\usepackage{manyfoot}%
\usepackage{booktabs}%
\usepackage{algorithm}%
\usepackage{algorithmicx}%
\usepackage{algpseudocode}%
\usepackage{listings}%
\usepackage{caption}
\usepackage{float} 
\usepackage{tikz}
\usepackage{threeparttable}
\usetikzlibrary{arrows.meta,calc}
\usepackage[hidelinks, colorlinks=true, linkcolor=blue, citecolor=blue, urlcolor=magenta]{hyperref}
\setlist[enumerate]{label=(\roman*)}
\def\A{\mathcal{A}}
\def\B{\mathcal{B}}

\def\I{\mathcal{I}}
\def\R{\mathbb{R}}
\def\E{\mathbb{E}}
\def\X{\mathbb{X}}
\def\S{\mathbb{S}}
\def\Q{\mathbb{Q}}
\def\D{\mathcal{D}}

\def\Diag{\operatorname{Diag}}
\def\Arw{\operatorname{Arw}}

\def\interior{\operatorname{int}}
\def\K{\mathbb{K}}

\def\O{\mathcal{O}}
\def\W{\mathcal{W}}
\def\H{\mathcal{H}}
\def\hE{\hat{\mathbb{E}}}

\def\h{\bar{h}}

\def\et{\eta_\rho(w;\mu)}
\def\dx{\Delta x}
\def\hx{\hat{x}}
\def\dhx{\Delta \hat{x}}
\def\bx{\bar{x}}
\def\ds{\Delta s}
\def\dlam{\Delta \lambda}

\renewcommand{\[}{\begin{equation}} 
\renewcommand{\]}{\end{equation}}   
\DeclareMathOperator*{\argmax}{arg\,max}
\DeclareMathOperator*{\argmin}{arg\,min}
\DeclareMathOperator*{\tr}{tr}

\journalname{}
\begin{document}

\title{Polynomial iteration complexity of a path-following smoothing Newton method for symmetric cone programming
\thanks{The authors are listed in alphabetical order.}
}

\titlerunning{Polynomial iteration complexity of a PFSNM for SCP}        

\author{Yu-Hong~Dai \and Ruoyu~Diao \and Xin-Wei~Liu  \and Rui-Jin~Zhang }

\authorrunning{Y.H. Dai, R. Diao, X.W. Liu,  R.J. Zhang} 

           \institute{ Yu-Hong~Dai $\cdot$ Ruoyu~Diao \at
              State Key Laboratory of Mathematical Sciences, Academy of Mathematics and Systems Science, Chinese Academy of Sciences, and the University of Chinese Academy of Sciences, Beijing, China \\         
              \email{dyh@lsec.cc.ac.cn, diaoruoyu18@mails.ucas.ac.cn}
           \and
           Xin-Wei~Liu
            \at
             Institute of Mathematics, Hebei University of Technology, Tianjin, China \\
             \email{mathlxw@hebut.edu.cn} 
              \and
               Rui-Jin~Zhang (Corresponding author) \at
              School of Mathematical Sciences and LPMC, Nankai University, Tianjin, China\\
              \email{zhangrj@nankai.edu.cn} 
}

\date{}

\makeatletter
\def\makeheadbox{}
\patchcmd{\@maketitle}{\hrule\@height0.35mm\noindent}{}{}{}
\makeatother
\maketitle



\begin{abstract}

 It has long remained open whether smoothing Newton methods (SNMs) for symmetric cone programming (SCP) admit polynomial iteration complexity. A key difficulty lies in the lack of an analogue of the self-concordant convex framework underlying interior-point methods (IPMs). In this paper, inspired by Nemirovski's self-concordant convex-concave theory, we address this open problem by introducing a reduced barrier augmented Lagrangian (BAL) function. We prove that the reduced BAL function is self-concordant convex-concave and establish that the parameterized smooth system arising in SNMs coincides with the first-order optimality conditions of an associated minimax problem. Motivated by this equivalence, we propose a path-following smoothing Newton method (PFSNM). The reduced BAL function induces a central path and an associated neighborhood, which provide estimates for the Newton decrement needed for the path-following analysis. As a result, the method achieves an iteration complexity of $\O(\sqrt{\nu}\ln(1/\varepsilon))$, matching the best-known short-step complexity for IPMs. Numerical results on standard benchmarks show that PFSNM is competitive with several well-known interior-point solvers, and the observed performance is consistent with the theoretical development.

\end{abstract}
\section{Introduction}\label{sec1}

Symmetric cone programming (SCP) is a fundamental class of convex optimization problems that includes linear programming (LP), second-order cone programming (SOCP), semidefinite programming (SDP), and their Cartesian products. Let $\E$ be a Euclidean Jordan algebra equipped with a bilinear operation ``$\circ$'' and an identity element $e$, and let $\K \subseteq \E$ be the associated symmetric cone, i.e., a closed convex cone that is both self-dual and homogeneous. We consider the standard primal-dual form of SCP:
\begin{equation}\label{SCP}
\begin{aligned}
(\operatorname{P}) \qquad&\min\, \left\{\langle {c}, {x}\rangle\,|\, \A {x}={b},\,x\in  \K\right\},\\
(\operatorname{D}) \qquad&\max \left\{\langle {b}, \lambda \rangle\,|\, \A^{*} \lambda+{s}={c},\,s\in  \K\right\},
\end{aligned}
\end{equation}
where $c,\,x,\,s\in\E$ and $b,\,\lambda\in \mathbb{R}^{m}$. The linear operator $\mathcal{A}:\E\rightarrow\mathbb{R}^{m}$ is assumed to be surjective, and $\mathcal{A}^{*}$ denotes its adjoint.
We assume that the Slater condition holds for both (P) and (D), which is a standard assumption in SCP.

Interior-point methods (IPMs) are among the most important algorithms for symmetric cone programming. They replace the complementarity condition $x\circ s = 0$ in the Karush--Kuhn--Tucker (KKT) system of~\eqref{SCP} by  $x\circ s = \mu e$ with $\mu >0$, which yields the perturbed KKT system
\begin{equation}
	\A x=b,\; \A^{*}\lambda+s=c,\; x,s\in \interior(\K),\;  x\circ s=\mu e.
\end{equation}
The solutions of this system form the central path. IPMs trace this path as $\mu \downarrow 0$. Their complexity theory is based on a self-concordant convex framework~\cite{nesterov1994interior}, which induces a local metric, yields estimates for the Newton decrement, and provides a natural way to define neighborhoods of the central path. In SCP, this framework leads to the classical polynomial iteration complexities: an $\O(\sqrt{\nu}\ln(1/\varepsilon))$ iteration complexity for short-step methods~\cite{de2002aspects,de2016turing,schmieta2003extension,vavasis1996primal,wright1997primal} and an $\O(\nu\ln(1/\varepsilon))$ iteration complexity for long-step methods~\cite{de2002aspects,nesterov1997long,nocedal2006numerical,schmieta2003extension,wright1997primal}, where $\nu$ denotes the rank of $\K$ and $\varepsilon$ is the target accuracy. These worst-case complexities explain why IPMs admit a remarkably robust global theory while retaining strong practical performance.

Alongside IPMs, smoothing Newton methods (SNMs) constitute another important class of algorithms for SCP. The basic idea is to reformulate the KKT system as a system of nonsmooth equations and then smooth the complementarity condition. This yields the parameterized smooth system
\begin{equation}
	\A x=b,\; \A^{*}\lambda+s=c,\; \Phi(x,s;\mu)=0, 
\end{equation}
where $\Phi$ is a chosen smoothing function and $\mu$ is driven to zero via a continuation strategy \cite{chen2003non,huang2004sub,kanzow1996some,peng1999non}. Common choices for $\Phi$ include the smoothing Fischer--Burmeister (FB) function~\cite{kanzow1996some,qi2000new} and the smoothing Chen--Harker--Kanzow--Smale (CHKS) function~\cite{chen1993non,kanzow1996some,smale2000algorithms}.   Unlike IPMs, SNMs do not require the iterates to remain strictly in the interior and are therefore sometimes called non-interior continuation methods. The first non-interior path-following method was proposed by Chen and Harker \cite{chen1993non}. Subsequently, Burke and Xu \cite{burke1998global} established the first global linear convergence result for a non-interior path-following method for linear complementarity problems, and later proved its local quadratic convergence under suitable assumptions \cite{burke2000non}. Qi, Sun, and Zhou \cite{qi2000new} gave a new formulation of smoothing Newton methods for nonlinear complementarity problems and box-constrained variational inequalities, thereby providing a unified framework that strongly influenced later developments of SNMs. For further advances and related results on SNMs, we refer the reader to \cite{chan2008constraint,huang2004sub,kanzow1999jacobian,kong2008regularized,liang2024squared,sun2004squared} and the references therein.

Despite these appealing  convergence results, SNMs have lacked a polynomial iteration complexity guarantee, which has remained a long-standing open problem (see Burke and Xu \cite{burke1998global} and Kanzow \cite{kanzow1996some}). A key difficulty is that classical SNMs do not admit an analogue of the self-concordant convex framework that underlies the polynomial complexity theory of IPMs. Such a framework provides a central path, a neighborhood structure defined by an appropriate merit function, and the Newton-decrement estimates required for path-following analysis. These are fundamental because they give estimates for the descent of the merit function when $\mu$ is updated, which are typically hard to obtain without the framework.

Several attempts have been made to address this problem. One line of research incorporates the parameterized smooth system into an interior-point framework. A representative example is the  interior-point path-following algorithm proposed by Xu and Burke \cite{xu1999polynomial},  which uses the smoothing CHKS function to generate a rescaled Newton direction within the interior-point framework and achieves a polynomial complexity. However, its iterates are still required to stay in the interior of the cone. In contrast, Hotta, Inaba, and Yoshise \cite{hotta2000complexity} proposed an SNM based on the smoothing CHKS function that eliminates the interior-point requirement, but with a non-polynomial complexity of $\O\left({\varepsilon^{-6}}\ln{\varepsilon^{-2}}\right)$, which is far from the standard polynomial iteration complexities of IPMs. Hence, the following central problem is still open:

\begin{quote}
\textit{Can smoothing Newton methods for symmetric cone programming attain polynomial iteration complexity?}
\end{quote}

In this paper, we answer this question affirmatively. Our approach is inspired by the self-concordant convex-concave framework introduced by Nemirovski~\cite{nemirovski1999self}, which provides a natural viewpoint for extending self-concordant ideas from convex minimization to saddle-point problems. The key step is to introduce a reduced barrier augmented Lagrangian (BAL) function that reveals the minimax structure underlying SNMs. We show that this reduced BAL function is self-concordant convex-concave, and that the parameterized smooth system associated with SNMs coincides with the first-order optimality conditions of a minimax problem with the reduced BAL function as the objective. This equivalence induces a local metric for SNMs, enabling the definition of a central path and an associated neighborhood analogous to those in the interior-point framework. More importantly, it provides the Newton decrement estimates required to control the path-following process. On this basis, we propose a path-following smoothing Newton method (PFSNM) for SCP and establish a worst-case iteration complexity of $\O(\sqrt{\nu}\ln(1/\varepsilon))$, which matches the best-known short-step complexity for IPMs on symmetric cones. To the best of our knowledge, this is the first polynomial iteration complexity result for a smoothing Newton method in the general SCP setting.

Although our main focus is theoretical, the resulting method is also computationally attractive. PFSNM admits a Newton system with an explicit Schur-complement structure, leading to an efficient system-formation procedure. Furthermore, numerical experiments on standard benchmark problems show that PFSNM is competitive with several well-known interior-point solvers. These results are consistent with the established polynomial-complexity theory.

\subsection{Organization}
 The remainder of the paper is organized as follows. Section~\ref{sec2} reviews preliminaries on Euclidean Jordan algebras and self-concordant convex-concave functions. Section~\ref{sec:equivalence characterization} introduces the reduced BAL function, establishes its self-concordant convex-concave property, and characterizes the parameterized smooth system via an equivalent minimax formulation. Section~\ref{sec:alg} presents the proposed path-following smoothing Newton method and the associated merit functions.
Section~\ref{sec5} analyzes the effect of updating the smoothing parameter on the merit functions and derives the polynomial iteration complexity of the proposed method. Numerical results are reported in Section~\ref{sec:computational-results} and illustrate the practical effectiveness of PFSNM. The paper concludes in Section~\ref{sec7}. Auxiliary technical proofs are
deferred to Appendix~\ref{secA1}.

\subsection{Notation}
Throughout the paper, $\hE$ and $\E$ denote finite-dimensional Euclidean spaces. For an integer $k\ge 1$, let $C^k(\E,\hE)$ denote the space of
$k$-times continuously differentiable functions from $\E$ to $\hE$; if $\hE=\mathbb{R}$, we write $C^k(\E):=C^k(\E,\mathbb{R})$. For $f\in C^{k} (\E)$, let $D^k f(x)[h_1, \dots, h_k]$ denote the $k$-th differential of $f$ at $x$ along directions $h_1, \dots, h_k \in {\E}$. Then $D^k f(x)$ is a symmetric $k$-linear form. In particular, $D^2 f(x):\E \to \E$ is also a linear operator, satisfying
$$
    \langle h_1,D^2 f(x)h_2 \rangle = D^2 f(x)[h_1,h_2], \quad \forall\, h_1,h_2\in \E.
$$
The gradient of $f$ at $x$ is denoted by $\nabla f(x)$ and is defined by
$$
\langle \nabla f(x), h \rangle = D f(x)[h], \quad \forall\, h\in \E.
$$
For $g\in C^k(\hE \times \E)$ and $w=(\hx,s)\in \hE\times \E$, we write
$\nabla_{\hx} g(w)$ and $\nabla_s g(w)$ for the partial gradients, $D_{\hx}g(w)$ and $D_s g(w)$ for the corresponding partial derivatives.
Let $\W,\, \H: \E \to \E$ be linear operators. We write $\H \succ \W$ if
$$
        \langle h, \H h \rangle > \langle h, \W h \rangle, \quad \forall\, h\in \E \setminus \{0\}.
$$
In particular, $\H \succ 0$ means $\langle h, \H h \rangle>0$ for all nonzero $h\in \E$. The interior of a cone ${\K}$ is denoted by ${\rm int}({\K})$. Additional notation will be introduced as needed.

\section{$\alpha$-self-concordant convex-concave functions}\label{sec2}

This section reviews the core concepts needed in the subsequent analysis. We begin by introducing Euclidean Jordan algebras, which provide the algebraic foundation for symmetric cones and thus play a central role in symmetric cone optimization. We then revisit and generalize the theory of $\alpha$-self-concordant convex-concave functions.

\subsection{Euclidean Jordan algebras and symmetric cones}\label{sec2.1}

\begin{definition}
    A Jordan algebra $(\E,\circ)$ is a finite-dimensional vector space over the real field $\R$ with a bilinear mapping $\circ:\E\times\E\to \E$ such that, for all $x,y,z\in\E$, 
    \begin{equation*}
    x\circ y = y\circ x,\ \, x^2\circ(x\circ y)=x\circ(x^2\circ y),\ \text{where } x^2:=x\circ x.
\end{equation*}
    The algebra is called Euclidean if there exists an inner product, denoted by $\langle \cdot, \cdot \rangle$, which is associative, i.e., for all $x,y,z\in\E$,
    \begin{equation*}
       \langle x\circ y,z\rangle=\langle x,y\circ z\rangle.
    \end{equation*}
\end{definition}

 A crucial property of Euclidean Jordan algebras is that every element admits a spectral decomposition, which generalizes the eigenvalue decomposition of a symmetric matrix.
\begin{theorem}\label{thm:spectral-decomposition}
Let $\E$ be a Euclidean Jordan algebra of rank $\nu$. 
For any $z \in \E$, there exist pairwise orthogonal primitive idempotents $\{v_{1},\ldots,v_{\nu}\}$ and unique real eigenvalues $\lambda_{1}(z),\ldots,\lambda_{\nu}(z)$ such that
\[
z = \sum_{i=1}^{\nu} \lambda_{i}(z) v_{i},
\]
where the idempotents satisfy
$$
\sum_{i=1}^\nu v_i = e, \ v_i \circ v_j = 0 \ \text{for all } i \neq j, \ \text{and} \ v_i \circ v_i = v_i \ \text{for all } i.
$$
\end{theorem}

By the spectral decomposition, the trace and determinant of $z$ are defined analogously to those of a real matrix: 
\[
    \tr(z) : = \sum\limits_{i=1}^\nu \lambda_i(z), \quad   \det (z):= \prod\limits_{i=1}^\nu \lambda_i (z).
\]
An element $z$ lies in the interior of the cone $\K$ if and only if all its eigenvalues are strictly positive; in particular, $\det(z)>0$. The spectral decomposition also enables a functional calculus on $\E$. Given a scalar function $g: \mathbb{R} \to \mathbb{R}$ and
an element $z\in \E$ with the spectral decomposition $z = \sum_{i=1}^\nu \lambda_i(z) v_i$, define 
\[
g(z) := \sum_{i=1}^\nu g(\lambda_i(z)) v_i.
\]
This definition allows for operations such as the square root $z^{1/2}$, the inverse $z^{-1}$, 
and the logarithm $\ln(z)$.  

By \cite[Proposition III.1.5]{faraut1994analysis}, the bilinear form $\tr (x\circ y)$ is symmetric, positive definite, and associative. Thus, $ \tr (x\circ y)$  defines an inner product. In the sequel, we denote this inner product by $\langle x,y\rangle$, i.e., 
\[
        \langle x, y \rangle : = \tr (x\circ y).
\]

In the following, we present three commonly used symmetric cones and their algebraic properties, which are essential to the development of our algorithm.

\begin{table*}[htbp]
\centering
\caption{Common types of symmetric cones.}
\label{tab:cones}

\resizebox{\linewidth}{!}{%
\begin{tabular}{@{}cccc@{}}
\toprule
\textbf{Cone} 
& \textbf{Mathematical representation} 
& \textbf{Jordan product} 
& \textbf{Spectral decomposition} \\
\midrule
Nonnegative orthant
& $\mathbb{R}_+^n = \{ x \in \mathbb{R}^n \mid x_i \geq 0,\ \forall\, i=1,\dots,n \}$
& $x \circ y := \Diag(x)y$
& $x = \sum_{i=1}^n x_i e_i$ \\
\midrule
Second-order cone
& $\mathbb{Q}^{n+1} = \left\{ (x_0; \bar{x}) \in \mathbb{R}\times\mathbb{R}^{n} \mid x_0 \geq \|\bar{x}\|_2 \right\}$
& $x \circ y := \Arw(x)y$\textsuperscript{a}
& $x = \lambda_1 v_1 + \lambda_2 v_2$\textsuperscript{b} \\
\midrule
Positive semidefinite cone
& $\mathbb{S}_+^n = \{ X \in \mathbb{R}^{n \times n} \mid X \succeq 0 \}$
& $X \circ Y := \frac{1}{2}(XY + YX)$
& $X = \sum_{i=1}^n \lambda_i v_i v_i^\top$\textsuperscript{c} \\
\bottomrule
\end{tabular}%
}

\vspace{1mm}
\begin{minipage}{\linewidth}
\footnotesize
\textsuperscript{a}\ $\Arw(x):=\begin{pmatrix}
x_0 & \bar{x}^{\top}\\
\bar{x} & x_0 I_{n\times n}
\end{pmatrix}$.

\textsuperscript{b}\ Let $\tilde{v} \in \mathbb{R}^{n}$ satisfy $\|\tilde{v}\|=1$. 
Then the eigenvalues of $x$ are
$\lambda_1=x_0+\|\bar{x}\|$ and $\lambda_2=x_0-\|\bar{x}\|$, and the corresponding primitive idempotents are
$$
	\begin{array}{ll}
		{v_1}=\left\{
		\begin{array}{ll}
			\dfrac{1}{2}\left(1;\dfrac{{\bar{x}}}{\|{\bar{x}}\|}\right), &\text{if\, }{\bar{x}}\neq 0;\\
			\dfrac{1}{2}\left(1;{\tilde{v}}\right), &\text{if\, }{\bar{x}}= 0,
		\end{array}
		\right. \
		&{v_2}=\left\{
		\begin{array}{ll}
			\dfrac{1}{2}\left(1;-\dfrac{{\bar{x}}}{\|{\bar{x}}\|}\right), &\text{if\, }{\bar{x}}\neq 0;\\
			\dfrac{1}{2}\left(1;-{\tilde{v}}\right), &\text{if\, }{\bar{x}}= 0.
		\end{array}
		\right.
	\end{array}
	$$

\textsuperscript{c}\ $\{\lambda_i\}_{i=1}^n$ are the eigenvalues of $X$, and
$\{v_i v_i^\top\}_{i=1}^n$ are the corresponding primitive idempotents.
\end{minipage}
\end{table*}

\subsection{$\alpha$-self-concordant convex-concave functions}\label{sec2.3}

In this subsection, we review the definition and basic properties of $\alpha$-self-concordant convex-concave functions. They provide the theoretical tools for analyzing Newton's method for finding saddle points. We begin with the classical notion of self-concordant convex functions.

\begin{definition}[{\cite[Definition~2.1]{nemirovski1999self}}]
Let $\E$ be a Euclidean Jordan algebra, $\X \subset \E$ be an open convex domain, and $\alpha>0$. A convex function $ f\in C^3(\X)$ is called $\alpha$-self-concordant on $\X$ if the following conditions hold: 
\begin{enumerate}
    \item $f$ is a barrier for $\X$, i.e., $f(x^{(k)}) \to \infty$ along every sequence of points $x^{(k)}\in \X$ converging to the boundary of $\X$.
    \item For all $x\in \X$ and $h_x \in \E$, 
    \[
\left| D^3 f(x)[h_x, h_x, h_x] \right| \leq \frac{2}{\alpha^{1/2}} \left( D^2 f(x)[h_x, h_x] \right)^{3/2}.
\]  
\end{enumerate} 
If $\alpha=1$, $f$ is called standard self-concordant. An $\alpha$-self-concordant convex function $f$ is said to be nondegenerate if  $D^2 f(x) \succ 0$ for all $x\in \X$.
\end{definition}

It is well known that an $\alpha$-self-concordant convex function satisfies a Dikin ellipsoid bound, which characterizes the local geometry induced by its second-order derivative.

\begin{theorem}[{\cite[Theorem 2.1.1]{nesterov1994interior}}]\label{theorem 1}
Let $\E$ be a Euclidean Jordan algebra and $\X \subset \E$
be an open convex domain. Let $f$ be an $\alpha$-self-concordant convex function on $\X$, $x\in \X$, and  $\dx \in \E$. If $r:=\sqrt{\frac{1}{\alpha} D^2 f(x)[\dx, \dx] }<1$, then $x+\dx \in \X$ and for all $h_x\in \E$, 
\[
(1- {r})^2 D^2f(x)[h_x,h_x] \le D^2f(x+\dx )[h_x,h_x] \le \frac{1}{(1-{r})^2}D^2f(x)[h_x,h_x].
\]
\end{theorem}

Each symmetric cone $\K$ admits a natural barrier function (see \cite[Section 2.6]{vieira2007jordan}) $\phi: \interior(\K) \to \R$, defined as 
\begin{equation*}
    \phi(x): = - \ln (\det (x)). 
\end{equation*}
By \cite{hauser2002self}, $\phi \in C^{\infty} (\interior(\K))$ is a nondegenerate standard self-concordant convex function on $\interior(\K)$. For every $x\in \interior(\K)$, the second-order derivative $D^2 \phi (x)\succ 0$. Consequently, the inverse operator $(D^2 \phi(x))^{-1}:\E \to \E$ is well defined and satisfies $(D^2 \phi(x))^{-1} \succ 0$. In particular, the gradient and second-order derivative of $\phi$ satisfy 
\begin{equation}\label{eq:derivatives of phi}
    \nabla \phi(x)=-x^{-1}, \quad \langle \nabla \phi(x), (D^2 \phi (x))^{-1} \nabla \phi(x) \rangle = \nu, \quad \forall\, x\in \interior(\K).
\end{equation}

The subsequent analysis focuses on $\alpha$-self-concordant convex-concave functions. We consider unconstrained minimax problems whose objective is convex in the minimization variables and concave in the maximization variables. To extend Newton's method to this setting, it is essential to identify a class of convex-concave functions that exhibits geometric properties similar to those of self-concordant convex functions. This motivates the concept of $\alpha$-self-concordant convex-concave functions, which generalizes the definition in \cite{nemirovski1999self}.

\begin{definition}\label{def:sccc function}
Let $\hE$ and $\E$ be finite-dimensional Euclidean spaces,
$
f(\hx, s)\in C^3 (\hE \times \E),
$
and $\alpha>0$.
The function $f$ is called $\alpha$-self-concordant convex-concave on $\hE \times \E$ if the following conditions hold:
\begin{enumerate}
    \item $f$ is convex in $\hx \in \hE$ for every $s \in \E$, and concave in $s \in \E$ for every $\hx \in \hE$.
    \item For every $w = (\hx, s) \in \hE \times \E$ and $h = (h_{\hx}, h_s) \in \hE\times \E$,
    \[
    \left| D^3 f(w)[h, h, h] \right| \leq \dfrac{2}{\alpha^{1/2}} \left(  S_f(w) [h,h] \right)^{3/2},  \]
    where
    $
    S_f(w)[h,h] : =
    D^2_{\hx \hx} f(w)[h_{\hx},h_{\hx}] - D^2_{ss} f(w)[h_s,h_s]. 
    $
\end{enumerate}
If $\alpha=1$, $f$ is called standard self-concordant convex-concave. An $\alpha$-self-concordant convex-concave function $f$ is called nondegenerate if the quadratic form $S_f (w) $ is positive definite for all $w\in \hE \times \E$.
\end{definition}

\begin{remark}
 The concept of an $\alpha$-self-concordant convex-concave function can also be defined on an open convex domain (see \cite{nemirovski1999self}). The only difference is that, in that setting, the functions $f(\cdot,s)$ and $-f(\hx,\cdot)$ are required to be barriers, respectively. In contrast, our analysis is carried out on the entire space $\hE \times \E$, which has no boundary. Therefore, no barrier property is required in our definition.
\end{remark}

The following proposition relates nondegenerate $\alpha$-self-concordant convex-concave functions to nondegenerate $\alpha$-self-concordant convex functions.

\begin{proposition}\label{proposition 2}
Let $\hE$ and $\E$ be finite-dimensional Euclidean spaces, and let $f(\hx,s)$ be a nondegenerate $\alpha$-self-concordant convex-concave function on $\hE\times  \E$. Then the following properties hold:
\begin{enumerate}
    \item For every $s\in \E$, $f(\cdot,s)$ is nondegenerate $\alpha$-self-concordant on $\hE$, and for every $\hx\in \hE$, $-f(\hx,\cdot)$ is nondegenerate $\alpha$-self-concordant on $\E$.
    \item For every $w=(\hx,s)\in \hE\times  \E $ and $h_1,\,h_2,\,h_3\in\hE\times \E$,
    $$
    \left| D^3 f(w)[h_1, h_2, h_3] \right| \leq \dfrac{2}{\alpha^{1/2}}\prod_{i=1}^3 \sqrt{  S_f(w) [h_i,h_i] }. 
$$
\end{enumerate}
\end{proposition}
\begin{proof}
The proof is provided in Appendix~\ref{proof:Prop 1}.
\end{proof}

Let $f$ be a nondegenerate  $\alpha$-self-concordant convex-concave function on $\hE\times \E$. For any $w=(\hx,s)\in \hE\times \E$ and any $h=(h_{\hx},h_s)\in \hE \times \E$, we define two local norms of $h$ associated with $f$ at $w$ by
\begin{equation}\label{def:norm-for-selfconcordant-convex-concave}
\begin{aligned}
\|h\|_{f,w,\alpha}:=\sqrt{\frac{1}{\alpha} S_f(w)[h,h]},\,\,
\|h\|^{*}_{f,w,\alpha}:=\sqrt{\frac{1} {\alpha} (S_f(w))^{-1}[h,h] },
\end{aligned}
\end{equation}
 where $(S_f (w) )^{-1} [h,h] : = \langle h_{\hx}, (D^2_{\hx \hx} f (w))^{-1} h_{\hx} \rangle - \langle h_s, (D^2_{ss} f(w))^{-1} h_s \rangle$.  Since $\alpha$ is fixed once $f$ is specified, we omit it for brevity and write $\|h\|_{f,w}$ and $\|h\|_{f,w}^*$ as shorthand for $\|h\|_{f,w,\alpha}$ and $\|h\|_{f,w,\alpha}^*$, respectively. 

We introduce three merit functions to measure how far the current iterate is from satisfying the optimality conditions. Specifically, let
\begin{equation}\label{def:merit-function-for-self-concordant-convex-concave-function}
\delta(w):=\|\Delta w\|_{f,w}, \,\, \xi(w):=\|\nabla f(w)\|^*_{f,w},\,\, \theta(w): = \max\limits_{\tilde{s}} f(\hx,\tilde{s}) - \min\limits_{\tilde{x}} f(\tilde{x},s),
\end{equation}
where
$\Delta w := -(D^2 f(w))^{-1}\nabla f(w) $. Following the notation in \cite{nemirovski1999self}, let $K(\theta) := \bigl\{w \mid \theta(w) < +\infty \bigr\}$.  For $w\in K(\theta)$, the optimization problems $\max_{\tilde{s}}f(\hx,\tilde{s})$ and  $\min_{\tilde{x}}f(\tilde{x},s)$ attain global optimal solutions, which are denoted by $s(\hx)$ and $\hx(s)$, respectively. We further define the merit functions:
\begin{equation}\label{merit-function-for-self-concordant-convex-functions}
    \tilde{\delta}_{\hx} (w) = \sqrt{\frac{1}{\alpha} \langle \widetilde{\Delta x}, D_{\hx \hx}^2 f(w) \widetilde{\Delta x}\rangle  },\,\, \tilde{\delta}_s(w)= \sqrt{ -\frac{1}{\alpha} \langle \widetilde{\Delta s}, D_{ss}^2 f(w) \widetilde{\Delta s}\rangle  },
\end{equation}
where $\widetilde{\Delta x}:=\hx-\hx(s)$ and $\widetilde{\Delta s}:=s-s(\hx)$. The connections among these merit functions are summarized in the following theorem.

\begin{theorem}\label{thm:properties-for-self-concordant-convex-concave-function}
    Let $\hE$ and $\E$ be finite-dimensional Euclidean spaces, and let $f: \hE\times  \E \to \mathbb{R}$ be a nondegenerate $\alpha$-self-concordant convex-concave function. For every $w=(\hx,s)\in \hE\times  \E $ and every $h=(h_{\hx},h_s)\in \hE\times\E$, the following statements hold: 
\begin{itemize}    
    \item[(i)] If $r=\|\Delta w\|_{f,w}<1$, then 
\[\label{ineq:bound of S}
(1- {r})^2 S_f(w)[h,h]\le S_f(w+ \Delta w)[h,h]\le \frac{1}{(1-{r})^2}S_f(w)[h,h].
\]
\item[(ii)] $\delta(w)\leq \xi(w) $. 
\item[(iii)] Let $w^+ :=w+\Delta w \in \hE\times \E$. If $\delta(w)<1$, then 
\[
\xi(w^+)\leq \left(\dfrac{\delta(w)}{1-\delta(w)}\right)^2.
\]
Furthermore, if $\delta(w) \le \xi (w) \leq 2-\sqrt{3}$, then $\xi(w^+)\leq \dfrac{\delta(w)}{2} \le \dfrac{\xi(w)}{2}$.
\item[(iv)] If $\xi(w)<\frac{1}{3}$, then 
\begin{equation}\label{eq:maxtdelta}
    \max\{ \tilde{\delta}_{\hx}(w),\tilde{\delta}_s(w)\} \le 1-(1-3 \xi(w))^{\frac{1}{3}}.
\end{equation}
Furthermore, if $\xi(w)\le 0.1$, then $ \max\{ \tilde{\delta}_{\hx}(w),\tilde{\delta}_s(w)\} < 0.2$.
\end{itemize}
\end{theorem}

\begin{proof}
The proofs of $(i)$--$(iii)$ follow directly from \cite[Propositions 2.3 and 5.1]{nemirovski1999self}. The proof of $(iv)$ is provided in Appendix~\ref{proof:Thm 3}.
\end{proof}

\section{A minimax reformulation of the smoothing Newton method}\label{sec:equivalence characterization}

In this section, we first reformulate the generalized parameterized smooth system as the first-order optimality conditions of a minimax problem. By eliminating the multiplier and the auxiliary variable, we obtain a reduced barrier augmented Lagrangian function $\eta_\rho$. We then show that $\eta_\rho$ is a nondegenerate $\mu$-self-concordant convex-concave function. Finally, we establish
that the generalized parameterized smooth system is equivalent to the first-order optimality conditions of the minimax problem defined by $\eta_\rho$, and that the Newton system of SNM is equivalent to the corresponding Newton system for that minimax problem.

\subsection{From the smoothing CHKS function to the reduced BAL function}

We begin with the smoothing CHKS function
$$\Phi(x,s;\mu)=x+s-\left((x-s)^2+4\mu e\right)^{1/2},$$
and then introduce its generalized form $\Phi_\rho$.

Given any point $(x,s)\in \E\times \E$ and $\mu>0$, the classical SNM \cite{engelke2002predictor,kanzow2002semidefinite,liu2006analysis} based on $\Phi$ for problem~\eqref{SCP} (inexactly) solves the parameterized smooth system
\begin{equation}\label{CHKS-SNM-nonlinear-equation}
    \A x  = b,\; \A^*\lambda+s = c,\; \Phi(x,s;\mu) = 0.
\end{equation}
Applying Newton's method to \eqref{CHKS-SNM-nonlinear-equation} yields the linearized system:
\begin{equation}\label{CHKS-SNM-Newton direction}
\begin{pmatrix}
    \A & 0 & 0\\
    0 & \I_{\E} & \A^*\\
    D_x \Phi(x,s;\mu) & D_s \Phi(x,s;\mu) & 0 
\end{pmatrix}
\begin{pmatrix}
    \dx\\
    \ds\\
    \dlam
\end{pmatrix}
= -
\begin{pmatrix}
    \A x - b\\
    \A^* \lambda + s -c\\ 
    \Phi (x,s;\mu)
\end{pmatrix},
\end{equation}
where $\I_{\E}: \E \to \E$ denotes the identity operator on $\E$. The classical SNM proceeds as follows. Starting from  $(x^{(0)},s^{(0)},\lambda^{(0)})$, it computes the Newton direction given by \eqref{CHKS-SNM-Newton direction}, performs a line search along this direction at each iteration, and progressively decreases $\mu$ toward zero.

To generalize the parameterized smooth system, we derive an equivalent characterization of  $\Phi$.
Let $\phi$ denote the natural barrier for $\K$.
Given a fixed $\rho > 0$, consider the following optimization problem:
\begin{equation}\label{proximal-z}
    \min_{z\in\interior(\K)}
	\Big\{\mu\phi(z)+\langle s,z\rangle+\dfrac{\rho}{2}\|z-x\|^2\Big\}.
\end{equation}
Here $\rho$ serves as the proximal regularization parameter in the quadratic term $\frac{\rho}{2}\|z-x\|^2$. For every fixed $\rho>0$, the associated proximal problem \eqref{proximal-z} is well defined and admits a unique solution. The special choice $\rho=1$ recovers the classical  smoothing CHKS function,
whereas allowing arbitrary $\rho>0$ yields a family of smoothing CHKS-type functions.

The unique solution of \eqref{proximal-z} is the proximal point of the proper closed convex function $\mu \phi (\cdot) + \langle s, \cdot \rangle$ at $x$; see \cite[Theorem 6.3]{beck2017first}. We denote this solution by $z_\rho(x,s;\mu)$. Then $$
z_\rho(x,s;\mu)=\frac{ \rho x-s + \left((\rho x-s)^2 + 4 \rho\mu e\right)^{1/2}  }{2\rho} \in \interior(\K),
$$ 
and satisfies the optimality condition
\begin{equation}
	\label{eq:z-stationary}
	\mu\nabla\phi(z_\rho(x,s;\mu)) + s + \rho(z_\rho(x,s;\mu)-x)=0.
\end{equation}
Furthermore, $z_\rho (x,s;\mu)$ is a function of $(x,s;\mu) \in \E\times \E\times \R_{++}$. For brevity, we write $z_\rho$ instead of $z_\rho(x,s;\mu)$ whenever no confusion arises.  The natural generalization of $\Phi$ is then defined by
\begin{equation}\label{def:Phi}
    \Phi_\rho(x,s;\mu): = 2(x - z_\rho(x,s;\mu) ), \quad \forall\, x,s \in \E,\ \mu>0,\ \text{and }\rho> 0. 
\end{equation}
In particular, $\Phi (x,s;\mu) = \Phi_1 (x,s;\mu)$.  The parameterized smooth system~\eqref{CHKS-SNM-nonlinear-equation} generalizes to
\begin{equation}
\label{SNM-nonlinear-equation}
\A x=b,\;
\A^*\lambda + s  =c,\;
\Phi_\rho(x,s;\mu)=0.
\end{equation}
Applying Newton's method to \eqref{SNM-nonlinear-equation} yields
\begin{equation}\label{SNM-Newton direction}
\begin{pmatrix}
    \A & 0 & 0\\
    0 & \I_{\E} & \A^*\\
    D_x \Phi_\rho (x,s;\mu) & D_s \Phi_\rho (x,s;\mu) & 0 
\end{pmatrix}
\begin{pmatrix}
    \dx\\
    \ds\\
    \dlam
\end{pmatrix}
= -
\begin{pmatrix}
    \A x -b\\
    \A^* \lambda + s -c\\ 
    \Phi_\rho (x,s;\mu)
\end{pmatrix}.
\end{equation}

The reformulation~\eqref{SNM-nonlinear-equation} extends the parameterized smooth system~\eqref{CHKS-SNM-nonlinear-equation}, with the classical smoothing  CHKS function recovered as the special case $\rho=1$. For simplicity, we continue to refer to the resulting method as an SNM.

We next show that \eqref{SNM-nonlinear-equation} can be naturally related to a minimax problem, thereby further clarifying the optimization interpretation of $\Phi_\rho$. Indeed, \eqref{SNM-nonlinear-equation} is equivalent to 
    \begin{equation}\label{minimax-KKT}
			\A x  =b,\; 
			c-s +  \frac{\rho}{2} \Phi_\rho (x,s;\mu)-\A^* \lambda = 0,\;
			\frac{1}{2}\Phi_\rho (x,s;\mu)  = 0.
\end{equation}
Combining \eqref{eq:z-stationary}, \eqref{def:Phi}, and \eqref{minimax-KKT} yields the system
\begin{equation*}
			\A x  =b,\; 
			 c-s - \rho(z - x) -\A^* \lambda = 0,\;
			z-x  = 0, \; \mu \nabla \phi(z) + s + \rho(z-x) = 0.
\end{equation*}
This system is precisely the system of first-order optimality conditions of the following minimax problem, whose objective is the \textit{barrier augmented Lagrangian} (BAL) function~$L_\rho$:
\begin{equation}\label{minimax-problem-1}
			\begin{aligned}
				\min\limits_{z \in \interior(\K),x}\max\limits_{s,\lambda} \ \left\{L_\rho(x,z,s,\lambda;\mu)\right\},\\
			\end{aligned}
		\end{equation}
where 
$
L_\rho(x,z,s,\lambda;\mu) := \langle {c}, {x}\rangle+\mu\phi(z)- \langle \lambda, \A x -b \rangle + \langle s,z-x\rangle+\dfrac{\rho}{2}\left\|z-x\right\|^2.$
The idea underlying the BAL function can be traced back to~\cite{liu2020globally}, and was further developed in~\cite{liu2022primal,liu2023novel,zhang2023iprqp,zhang2024iprsdp,zhang2026iprsocp}.

Reformulation~\eqref{minimax-problem-1} reveals that the parameterized smooth
system~\eqref{SNM-nonlinear-equation} can be interpreted as the first-order
optimality conditions of the minimax problem~\eqref{minimax-problem-1}. This suggests studying the SNM through the BAL function $L_\rho$. However,
$L_\rho$ is not suitable for direct self-concordant convex-concave analysis.
Indeed, it is degenerate in both the multiplier $\lambda$ and the $s$-variable, since $
D_{\lambda\lambda}^2 L_\rho(x,z,s,\lambda;\mu)=0
$
and
$
D_{ss}^2 L_\rho(x,z,s,\lambda;\mu)=0.
$ 
Therefore, we first restrict the formulation to the affine constraint $\A x=b$. Since $\A$ is surjective, the linear system $\A x = b$ admits a solution for the given $b$. Fix an arbitrary feasible point $\bar{x}$ such that $\A\bar{x}=b$. Let $\hE$ be a finite-dimensional Euclidean space such that 
$
{\rm dim}\, \hE = {\rm dim}\, {\rm ker}\, \A.
$
Then there exists an injective linear operator $\B : \hE \rightarrow \E$ such that 
$$
\A \B=0, \ \quad \B^* \B=\I_{\hE}.
$$  
Every feasible point $x$ satisfying $\A x = b$ can be written uniquely as 
$$
x=\bar{x}+\B \hat{x}, \quad \hat{x}\in \hE.
$$ 
Under this parametrization, the constraint term
$-\langle \lambda,\A x-b\rangle$ vanishes identically, and the multiplier
$\lambda$ no longer appears. Thus, the degeneracy with respect to $\lambda$ is removed.

We then eliminate the auxiliary variable
$z$. This reduction also removes the degeneracy in the $s$-variable and leads to a reduced
formulation suitable for a self-concordant convex-concave analysis. Define  $\eta_\rho (\cdot, \cdot; \mu) : \hE \times \E \rightarrow \R$ by 
\begin{equation}\label{def:eta}
\begin{aligned}
         \eta_\rho(\hx,s;\mu): =  \min\limits_{z\in \interior(\K)} \left\{ L_\rho (\bx + \B \hx,z,s,\lambda;\mu) \right\}.
\end{aligned}
\end{equation}
The value in $\eta_\rho(\hx,s;\mu)$ is independent of $\lambda$, because
$\A(\bar x+\B\hat x)=b$. Writing $x=\bar x+\B\hat x$, the unique minimizer in
\eqref{def:eta} is $z_\rho(x,s;\mu)$. Hence
\[
\eta_\rho(\hat x,s;\mu)
=
\langle c,x\rangle
+\mu\phi(z_\rho(x,s;\mu))
+\langle s,z_\rho(x,s;\mu)-x\rangle
+\frac{\rho}{2}\|z_\rho(x,s;\mu)-x\|^2 .
\]
We call $\eta_\rho$ the \textit{reduced BAL} function. In the sequel, we write 
$$
    w:=(\hx,s) \in \hE \times \E
$$
for the reduced minimax variable whenever no confusion can arise.
The resulting reduced minimax problem is
\begin{equation}\label{minimax-eta}
    \min_{\hat x\in\hE}\max_{s\in\E}
    \left\{ \eta_\rho(w;\mu) \right\}.
\end{equation}
Compared with the original BAL function $L_\rho$, the reduced function
$\eta_\rho$ removes the affine multiplier degeneracy and, after eliminating
$z$, has a nondegenerate curvature structure in both the minimization and
maximization variables. These structural properties are established in the next
subsection. In the remainder of the paper, we work with the reduced minimax
problem~\eqref{minimax-eta} and the function $\eta_\rho$.

\subsection{Properties of the reduced BAL function}

In this subsection, we discuss the properties of the reduced BAL function $\eta_\rho$. Recall that $z_\rho(x,s;\mu)$ satisfies
\begin{equation}\label{z-equation}
	\mu\nabla \phi(z_\rho(x,s;\mu))+s+\rho(z_\rho(x,s;\mu)-x)=0.
\end{equation}
Define the adjoint variable associated with $z_\rho(x,s;\mu)$ by
\begin{equation}\label{y-equation}
y_\rho(x,s;\mu):=s+\rho(z_\rho(x,s;\mu)-x).
\end{equation}
The variable $y_\rho (x,s;\mu)$ is also a function of $(x,s;\mu)\in \E\times \E\times \R_{++}$. Furthermore,
both $z_\rho(x,s;\mu)$ and $y_\rho(x,s;\mu)$ satisfy the following properties.
\begin{lemma}\label{Lemma:z-x=0}
For any scalars $\mu>0$ and $\rho > 0$, the following statements are equivalent:
\begin{equation}
\text{(i)}\; z_\rho(x,s;\mu) = x,\,
\text{(ii)}\; y_\rho(x,s;\mu) = s,\,
\text{(iii)}\; x,\, s\in \interior(\K),\, x \circ s = \mu e.
\end{equation}
\end{lemma}
\begin{proof}
The equivalence $(i) \Longleftrightarrow (ii)$ follows directly from the definition of $y_\rho(x,s;\mu)$. It suffices to prove that $(i) \Longleftrightarrow (iii)$. 

Suppose that $(iii)$ holds. By \eqref{z-equation},
\begin{equation*}
      z_\rho(x,s;\mu)  = \frac{\rho x - s + \left(   (s-\rho x)^2 + 4\rho \mu e\right)^{1/2}}{2\rho}.\\
\end{equation*}
Since $x\circ s = \mu e$, 
\begin{equation*}
\begin{aligned}
      z_\rho(x,s;\mu) &= \frac{\rho x - s + \left(   s^2 - 2\rho s \circ x + \rho^2 x^2 + 4\rho \mu e \right)^{1/2}}{2\rho} = x,\\
\end{aligned}
\end{equation*}
which establishes $(i)$.

Conversely, assume that $(i)$ holds. By \eqref{z-equation} and the inclusions $z_\rho, x\in \interior(\K)$, 
\begin{equation}\label{eq:nablaz-phi+s=0}
    \mu \nabla \phi(x) + s = 0.
\end{equation}
Combining \eqref{eq:derivatives of phi} and \eqref{eq:nablaz-phi+s=0} yields $s \in \interior(\K)$ and $x\circ s = \mu e$. This establishes $(iii)$ and completes the proof.
\end{proof}

Define the linear operators
\begin{equation}
{\W}=\dfrac{\mu}{\rho}D^2\phi(z_\rho(x,s;\mu)), \quad {\H}=\I_{\E}+\W.
\end{equation} 
Since the natural barrier $\phi$ is strictly convex on $\interior(\K)$, we have 
$$
\H \succ \W \succ 0, \quad \H \succ \I_{\E}.
$$
For brevity, all derivatives with respect to $\mu$ are denoted by a prime.
For example, $z_\rho'(x,s;\mu) := D_\mu z_\rho(x,s;\mu)$.
When no ambiguity arises, we also abbreviate $y_\rho(x,s;\mu)$
as $y_\rho$. The following theorem characterizes the derivatives of $z_\rho (x,s;\mu)$ and $y_\rho(x,s;\mu)$.
\begin{theorem}\label{thm:differentiability}
	For any $\rho> 0$, the functions  $z_\rho(x,s;\mu)$ and $y_\rho(x,s;\mu)$ are smooth with respect to $(x,s,\mu)$ on $\E\times\E\times\R_{++}$. Furthermore, their partial derivatives with respect to $(x,s)$ are given by
    \begin{equation}
        \begin{array}{llll}
           & D_x z_\rho(x,s;\mu) =\H^{-1},
			&	D_x y_\rho(x,s;\mu) =-\rho \H^{-1}\W,\\ 
            & D_s z_\rho(x,s;\mu) =-\rho^{-1}\H^{-1},
			 & D_s y_\rho(x,s;\mu) =\H^{-1}\W,\\ 
        \end{array}
    \end{equation}
    and the derivatives with respect to $\mu$ are given by 
    \begin{equation}
        \begin{aligned}
            &z'_\rho(x,s;\mu) =-\rho^{-1}\H^{-1}\nabla\phi(z_\rho(x,s;\mu)),\\
		& y'_\rho(x,s;\mu) =-\H^{-1}\nabla \phi(z_\rho(x,s;\mu)).
        \end{aligned}
    \end{equation}
\end{theorem}
\begin{proof}
 Recall that $\phi \in C^\infty(\interior(\K))$ and that $z_\rho(x,s;\mu)$ is defined as the unique solution to~\eqref{z-equation}. Since
\[
\mu D^2 \phi(z_\rho(x,s;\mu)) + \rho \I_{\E} = \rho \H \succ 0, \quad \forall\, (x,s,\mu) \in \E\times \E \times \R_{++},
\]
 the derivative of the mapping 
 $$
    F_\rho(x,z,s;\mu) = \mu \nabla \phi(z) + s + \rho (z-x) \in C^\infty (\E \times \E \times \E\times \R_{++}, \E)
 $$
 with respect to $z$ at $z=z_\rho (x,s;\mu)$ is nonsingular. Hence, the implicit function theorem implies $z_\rho\in C^\infty (\E \times \E\times \R_{++},\E)$.
By definition, 
\begin{equation}\label{equation 1}
	y_\rho(x,s;\mu)=s+\rho(z_\rho(x,s;\mu)-x), 
\end{equation}
 which implies that $y_\rho\in C^\infty(\E \times \E\times \R_{++},\E)$. Moreover, $y_\rho$ satisfies
 \begin{equation}\label{equation 2}
 	\mu\nabla \phi(z_\rho(x,s;\mu))+y_\rho(x,s;\mu)=0.
 \end{equation}
  Differentiating both sides of \eqref{equation 1} and \eqref{equation 2} with respect to $x$  yields 
  \begin{equation*}
  	D_x y_\rho(x,s;\mu)=\rho(D_x z_\rho(x,s;\mu)-\I_{\E})
  \end{equation*}
and 
 \begin{equation*}
	\mu D^2\phi(z_\rho(x,s;\mu)) D_x z_\rho(x,s;\mu)+D_x y_\rho(x,s;\mu)=0.
\end{equation*}
Consequently, we have
\begin{equation*}
	D_x z_\rho(x,s;\mu)=\rho(\mu D^2\phi(z_\rho(x,s;\mu))+\rho \I_{\E})^{-1}=\H^{-1}
\end{equation*}
 and 
 \begin{equation*}
 \begin{aligned}
       	D_x y_\rho(x,s;\mu) &=-\rho \left(\dfrac{\mu}{\rho}D^2\phi(z_\rho(x,s;\mu))+\I_{\E}\right)^{-1} \left(\dfrac{\mu}{\rho}D^2\phi(z_\rho(x,s;\mu))\right)\\
       	& =-\rho \H^{-1}\W.
 \end{aligned}
 \end{equation*}
The remaining identities can be proved in the same way.
\end{proof}

By Theorem~\ref{thm:differentiability} and \eqref{SNM-Newton direction}, the search direction generated by the SNM is equivalently written as the solution of the linear system
\begin{equation}\label{SNM-Newton-direction-rho}
    \begin{pmatrix}
    \A & 0 & 0\\
    0 & \I_{\E} & \A^*\\
    \H^{-1}\W & \rho^{-1} \H^{-1} & 0
    \end{pmatrix}
    \begin{pmatrix}
    \dx\\
    \ds\\
    \dlam
    \end{pmatrix}
    = -
    \begin{pmatrix}
    \A x -b \\
    \A^* \lambda + s -c\\
    x - z_\rho
    \end{pmatrix}.
\end{equation}
The following proposition guarantees the uniqueness of this search direction.

\begin{proposition}\label{prop:uniqueness-of-SNM-Newton}
For any point $(x,s,\lambda)\in \E \times \E \times \R^m$ and scalars $\mu>0$, $\rho > 0$, the Newton system~\eqref{SNM-Newton-direction-rho} generated by the SNM admits a unique solution.
\end{proposition}
\begin{proof}
The third equation in \eqref{SNM-Newton-direction-rho} gives 
\begin{equation*}
    \ds = - \rho \W \dx - \rho \H (x-z_\rho).
\end{equation*}
Substituting this expression into the remaining equations, it suffices to verify the uniqueness of the solution to
\begin{equation}\label{SNM-unique-Shur-complement}
    \begin{pmatrix}
    -\rho \W & \A^* \\
    \A & 0
    \end{pmatrix}
    \begin{pmatrix}
    \dx\\
    \dlam 
    \end{pmatrix}
    =- \begin{pmatrix}
    \rho \H (z_\rho-x)  + \A^* \lambda + s -c \\
    \A x - b 
    \end{pmatrix}.
\end{equation}
Since $-\rho \W$ is invertible, the Schur complement of $\begin{pmatrix}
    -\rho \W & \A^* \\
    \A & 0
    \end{pmatrix}$ relative to $-\rho \W$ is $\rho^{-1} \A \W^{-1} \A^*$. Moreover, $\W^{-1} \succ 0$ and $ \A$ is surjective. Therefore, $\rho^{-1} \A \W^{-1} \A^*$ is invertible and the system~\eqref{SNM-unique-Shur-complement} admits a unique solution. 
\end{proof}

The next corollary provides explicit formulas for  the first- and second-order derivatives of 
$\et$ with respect to $\hx$ and $s$.
\begin{corollary}\label{coro:differentiability-and-derivatives-of-eta}
	For any $\rho > 0$, the reduced BAL function $\et$ is smooth on $\hE\times \E\times \R_{++}$. 
    Furthermore, for $x=\bx+\B\hx$,
    \begin{equation}
    \nabla_{\hx} \et = \B^* (c-y_\rho(x,s;\mu)),\,\nabla_s \et =z_\rho(x,s;\mu)-x,
    \end{equation}
    and
    \begin{equation}
    \begin{aligned}
    D^2_{\hx \hx} \et &=\rho \B^* \H^{-1}\W \B,\quad          D^2_{ss} \et =-\rho^{-1}\H^{-1},\\
    D^2_{\hx s}\et
&=
-\B^*\H^{-1}\W,\quad D^2_{s\hx}\et
=
-\H^{-1}\W\B.
    \end{aligned}
    \end{equation}
\end{corollary}
\begin{proof}
	The smoothness of $\eta_\rho$ follows immediately from Theorem~\ref{thm:differentiability}. By differentiating \eqref{def:eta} with respect to $\hx$ and using the chain rule, we obtain
	\begin{equation*}
		\begin{aligned}
			 \nabla_{\hx} \et 
            &= \B^* ( \nabla_{x} \et )\\
            &=  \B^* \big( c+\mu D_x z_\rho(x,s;\mu)\nabla \phi(z_\rho(x,s;\mu)) \\
            & \qquad + (D_x z_\rho(x,s;\mu)-\I_{\E})\left( s + \rho(z_\rho(x,s;\mu)-x) \right) \big) .
		\end{aligned}
	\end{equation*}
    By \eqref{z-equation} and \eqref{y-equation}, we have
    	\begin{equation}\label{equation 5}
		\begin{aligned}
			 \nabla_{\hx} \et 
			&=  \B^* \big( D_x z_\rho(x,s;\mu)(\mu\nabla \phi(z_\rho(x,s;\mu))+ y_\rho(x,s;\mu)) \\
			& \qquad +c-y_\rho(x,s;\mu) \big)\\
			&= \B^* ( c-y_\rho(x,s;\mu) ).
		\end{aligned}
	\end{equation}
	  Further differentiating \eqref{equation 5} with respect to $\hx$ yields
 \begin{equation*}
 	D^2_{\hx \hx} \et=\rho \B^* \H^{-1}\W \B.
 \end{equation*}
The remaining identities follow by analogous arguments.
\end{proof}

Corollary~\ref{coro:differentiability-and-derivatives-of-eta} implies that $\et$ is convex in $\hx$ and concave in $s$. The following theorem further shows that $\et$ is a nondegenerate 
$\mu$-self-concordant convex--concave function.

\begin{theorem}\label{thm:self-concordant-of-eta}
	For any $\mu>0$ and $\rho > 0$,  the reduced BAL function $\eta_\rho(\cdot,\cdot;\mu)$ is a nondegenerate $\mu$-self-concordant convex-concave function on $\hE\times \E$. Furthermore, $\eta_\rho(\cdot,s;\mu)$ is nondegenerate $\mu$-self-concordant on $\hE$ for every $s\in \E$, and $-\eta_\rho(\hx,\cdot;\mu)$ is nondegenerate $\mu$-self-concordant on $\E$ for every $\hx \in \hE$. 
\end{theorem}
\begin{proof}
	The smoothness of $\eta_\rho$ follows from Corollary~\ref{coro:differentiability-and-derivatives-of-eta}. For any $h=(h_{\hx},h_s)\in \hE\times\E$,
    \begin{equation*}
       S_{\eta_\rho} (w;\mu)[h,h] = \rho \langle h_{\hx}, \B^* \H^{-1}\W \B h_{\hx}\rangle + \rho^{-1}\langle h_s, \H^{-1} h_s \rangle.
    \end{equation*}
    Since $\B$ is injective and $\H^{-1}\W \succ 0$, the operator $\B^* \H^{-1}\W\B$ is positive definite on $\hE$. Combined with $\H^{-1}\succ 0$, this shows that $S_{\eta_\rho} (w;\mu)$ is positive definite, so $\et$ is nondegenerate. By Definition~\ref{def:sccc function}, it suffices to prove that 
    \begin{equation*}
        | D^3 \et [h,h,h] | \le \dfrac{2}{\sqrt{\mu }} (  S_{\eta_\rho} (w;\mu)[h,h] )^{3/2}.
    \end{equation*}
    Let $\h = (\h_x,\h_s) = (\H^{-1}\B h_{\hx},\H^{-1} h_s) \in \E \times \E$. By Corollary~\ref{coro:differentiability-and-derivatives-of-eta} and the formula $\H = \I_{\E} + \W = \I_{\E} + \frac{\mu}{\rho} D^2 \phi(z_\rho)$, we have
    \begin{equation}\label{S_eta inequality}
    \begin{aligned}
         & S_{\eta_\rho} (w;\mu)[h,h]\\
          =&\ \rho \langle \h_x, \H\W \h_x\rangle + \rho^{-1} \langle \h_s, \H \h_s \rangle\\
        =&\  \mu  D^2 \phi(z_\rho) [\h_x,\h_x] + \frac{\mu^2}{\rho} \Vert  D^2 \phi(z_\rho) \h_x \Vert^2 + \rho^{-1} \Vert \h_s \Vert^2 + \frac{\mu}{\rho^2} D^2 \phi (z_\rho) [\h_s,\h_s]   \\
        =&\ \mu  D^2 \phi(z_\rho)[\h_x-\rho^{-1}\h_s, \h_x-\rho^{-1}\h_s]
         +\dfrac{1}{\rho}    \Vert \mu D^2 \phi (z_\rho)\h_x + \h_s \Vert^2 \\
        \ge &\ \mu  D^2 \phi(z_\rho)[\h_x-\rho^{-1}\h_s, \h_x-\rho^{-1}\h_s],
    \end{aligned}
    \end{equation}
    and 
    \begin{equation*}
    \begin{aligned}
         D^2 \et [h,h]  & = \rho \langle \B h_{\hx}, (\I_{\E}-\H^{-1}) \B h_{\hx} \rangle - \rho^{-1} \langle h_s, \H^{-1} h_s \rangle\\
         & \qquad \qquad + 2 \langle \B h_{\hx}, (\H^{-1}-\I_{\E})h_s\rangle.\\
    \end{aligned}
    \end{equation*}
    A direct calculation gives 
        \begin{subequations}\label{H-derivative}
            \begin{align}
              & D_{\hx} \H^{-1} [h_{\hx},\cdot,\cdot] =  - \frac{\mu}{\rho}  D^3 \phi (z_\rho)[\h_x,\H^{-1} \cdot,\H^{-1} \cdot],\\
              & D_s \H^{-1} [h_s,\cdot,\cdot] =   \frac{\mu}{\rho^2}  D^3 \phi (z_\rho)[\h_s,\H^{-1}\cdot,\H^{-1} \cdot].
            \end{align}
        \end{subequations}
 Thus, we have 
 \begin{equation*}
     \begin{aligned}
         & D^3 \et [h,h,h]\\  = &\  -\rho D_{\hx}  \H^{-1}[h_{\hx}, \B h_{\hx}, \B h_{\hx}]-\rho D_s \H^{-1}[h_s, \B h_{\hx}, \B h_{\hx}]\\
         &\  \qquad - \rho^{-1}D_{\hx} \H^{-1}[h_{\hx},h_s,h_s]-\rho^{-1} D_s \H^{-1}[h_s,h_s,h_s]\\
         &\ \qquad  +2D_{\hx} \H^{-1}[h_{\hx},\B h_{\hx},h_s]+2 D_s \H^{-1}[h_s,\B h_{\hx},h_s]\\
        = &\   \mu D^3 \phi(z_\rho)[\h_x-\rho^{-1}\h_s,{\h_x}-\rho^{-1}\h_s,\h_x-\rho^{-1}\h_s].
     \end{aligned}
 \end{equation*}
 Taking absolute values and using the standard self-concordance of $\phi$, we obtain
  \begin{equation*}
     \begin{aligned}
     \left|D^3 \et [h,h,h]\right|=&\ \mu\left| D^3 \phi(z_\rho)[\h_x-\rho^{-1}\h_s,{\h_x}-\rho^{-1}\h_s,\h_x-\rho^{-1}\h_s]\right|\\
     \le &\ {2\mu} \left\{ D^2 \phi(z_\rho)[\h_x-\rho^{-1}\h_s, \h_x-\rho^{-1}\h_s]\right\}^{3/2}\\
        \le &\ \dfrac{2}{\sqrt{\mu}} \left( 
        S_{\eta_\rho} (w;\mu)[h,h] \right)^{3/2},
     \end{aligned}
 \end{equation*}
    where the second inequality follows from \eqref{S_eta inequality}. Hence, $\eta_\rho (\cdot,\cdot;\mu)$ is a nondegenerate $\mu$-self-concordant convex-concave function on $\hE\times \E$. The remaining statements follow immediately from Proposition~\ref{proposition 2}.
\end{proof}

\subsection{Equivalent characterization of the parameterized smooth system}

This subsection establishes the key equivalence between the parameterized smooth system~\eqref{SNM-nonlinear-equation} and
the first-order optimality conditions of the minimax problem~\eqref{minimax-eta}. 

We begin by recalling the barrier subproblem arising in   IPMs:
\begin{equation}\label{barrier-problem-without-z}
    \min \left\{\langle {c}, {x}\rangle + \mu \phi(x) \,|\, \A x = b,\, x\in  \interior(\K)\right\}.
\end{equation}
The following lemma establishes the connection between \eqref{SNM-nonlinear-equation} and \eqref{barrier-problem-without-z}.
\begin{lemma}\label{lemma:SNM-equivalent-smooth-KKT}
    For any $\mu>0$ and $\rho > 0$, the triple $(x(\mu),s(\mu),\lambda(\mu))$ solves the parameterized smooth system~\eqref{SNM-nonlinear-equation} if and only if it is the KKT triple of the barrier subproblem~\eqref{barrier-problem-without-z}.
\end{lemma}
\begin{proof}
    By definition, the triple $(x(\mu),s(\mu),\lambda(\mu))$ solves \eqref{SNM-nonlinear-equation} if and only if it satisfies
    \begin{equation*}
\A x(\mu)-b=0,\;
\A^*\lambda(\mu) + s(\mu) - c =0,\;
\Phi_\rho (x(\mu),s(\mu);\mu)= 0.
\end{equation*}
Since  $\Phi_\rho(x(\mu),s(\mu);\mu) = 2(x(\mu)-z_\rho(x(\mu),s(\mu);\mu))$, it follows from Lemma~\ref{Lemma:z-x=0} that 
    \begin{equation*}
    \begin{aligned}
         \A x(\mu)  = b,\, \A^* \lambda(\mu) + s(\mu) - c  = 0,\, x(\mu),\,s(\mu) \in \interior(\K),\,
         x(\mu) \circ s(\mu) & = \mu e.
    \end{aligned}
\end{equation*}
Consequently, $(x(\mu),s(\mu),\lambda(\mu))$ is the KKT triple associated with the barrier subproblem~\eqref{barrier-problem-without-z}. The converse implication follows by reversing the above arguments, and the proof is complete.
\end{proof}

Based on Lemma~\ref{lemma:SNM-equivalent-smooth-KKT}, we now establish the equivalence between the parameterized smooth system~\eqref{SNM-nonlinear-equation} and the first-order optimality conditions of the minimax problem~\eqref{minimax-eta}.

\begin{theorem}\label{thm:equivalent-barrier-eta}
For any $\mu>0$ and $\rho > 0$, suppose that $(x(\mu),s(\mu),\lambda(\mu))$ solves the parameterized smooth system~\eqref{SNM-nonlinear-equation}. Then the pair $(\hx(\mu),s(\mu)):= ( \B^* (x(\mu) - \bar{x}), s(\mu))$ is a saddle point of the minimax problem \eqref{minimax-eta}.
Conversely, if $(\hx(\mu),s(\mu))$ is a saddle point of \eqref{minimax-eta}, then 
$$
(x(\mu),s(\mu),\lambda(\mu)):= (\bar{x}+ \B \hx(\mu), s(\mu), (\A \A^*)^{-1}\A (c-s(\mu)))
$$ solves the parameterized smooth system~\eqref{SNM-nonlinear-equation}.
\end{theorem}

\begin{proof}
Suppose that the triple $(x(\mu),s(\mu),\lambda(\mu))$ solves \eqref{SNM-nonlinear-equation}. By Lemma~\ref{lemma:SNM-equivalent-smooth-KKT}, we have 
\begin{equation}\label{thm:equivalent-barrier-eta-equation-1}
    \begin{aligned}
         \A x(\mu)  = b,\, \A^* \lambda(\mu) + s(\mu) - c  = 0,\, x(\mu),\,s(\mu) \in \interior(\K),\,
         x(\mu) \circ s(\mu) & = \mu e.
    \end{aligned}
\end{equation}
Together with Lemma~\ref{Lemma:z-x=0}, this yields
\begin{equation*}    
z_\rho(x(\mu),s(\mu);\mu) = x(\mu).
\end{equation*}
Thus, it follows from Corollary~\ref{coro:differentiability-and-derivatives-of-eta} that
\begin{equation*}
\nabla_s \eta_\rho (\hx(\mu),s(\mu);\mu) = z_\rho(x(\mu),s(\mu);\mu) - x(\mu) = 0.
\end{equation*}
Furthermore, 
\begin{equation*}
    \begin{aligned}
         \nabla_{\hx} \eta_\rho (\hx (\mu), s(\mu);\mu) & = \B^*(c - y_\rho(x(\mu),s(\mu);\mu))\\
         & = \B^*(c - s(\mu) - \rho (z_\rho(x(\mu),s(\mu);\mu)- x(\mu)) ).
    \end{aligned}
\end{equation*}
By the identity $z_\rho(x(\mu),s(\mu);\mu) =  x(\mu)$ and the second equation in \eqref{thm:equivalent-barrier-eta-equation-1}, we have 
\begin{equation*}
    \nabla_{\hx} \eta_\rho (\hx (\mu), s(\mu);\mu) = \B^*(c - s(\mu)) = \B^* \A^* \lambda(\mu)=0. 
\end{equation*}
Therefore, $(\hx(\mu),s(\mu)):=( \B^* (x(\mu) - \bar{x}), s(\mu))$ satisfies the first-order optimality conditions of the minimax problem~\eqref{minimax-eta}. Since $\eta_\rho$ is convex-concave, this pair is a saddle point of \eqref{minimax-eta}.

Conversely, suppose that $(\hx(\mu),s(\mu))$ is a saddle point of
\eqref{minimax-eta}. Then it satisfies 
\begin{equation}\label{thm:equivalent-barrier-eta-equation-2}
    \begin{aligned}
         z_\rho(x(\mu),s(\mu);\mu)  = x(\mu),\,
          \B^*(c - y_\rho(x(\mu),s(\mu);\mu)) =0,
    \end{aligned}
\end{equation}
where $x(\mu) = \bar{x} + \B \hx(\mu)$. By Lemma~\ref{Lemma:z-x=0}, we have 
$$
    x(\mu),\, s(\mu) \in \interior(\K),\; x(\mu) \circ s(\mu) = \mu e,
$$
and $\A x(\mu)-b = \A \bar{x} - b + \A \B \hx(\mu) = 0$. It remains to show that
\begin{equation*}
    \A^* \lambda(\mu) + s(\mu) - c = 0.
\end{equation*}
Note that the second equation of \eqref{thm:equivalent-barrier-eta-equation-2} implies 
$$
c - y_\rho(x(\mu),s(\mu);\mu) = c - s(\mu) \in ({\rm ker}\, \A)^\perp. 
$$
Since $\A$ is surjective and $\A^* (\A \A^* )^{-1} \A$ is the orthogonal projector onto $({\rm ker}\, \A)^\perp$, we obtain
\begin{equation*}
    \begin{aligned}
         \A^* \lambda(\mu) + s(\mu) - c  = 
         (\I_{\E} - \A^* (\A \A^* )^{-1} \A ) ( s(\mu) - c) = 0,
    \end{aligned}
\end{equation*}
which concludes the proof.
\end{proof}

\begin{remark}\label{rmk:existence-of-saddle-point}
 According to \cite[Theorem 4.1]{nesterov1998primal}, the barrier subproblem~\eqref{barrier-problem-without-z} admits a unique primal-dual optimal solution. Together with Theorem~\ref{thm:equivalent-barrier-eta}, this implies that the minimax problem~\eqref{minimax-eta} admits a unique saddle point for any $\mu>0$ and $\rho > 0$.
\end{remark}

The following theorem characterizes the search direction of the SNM.

\begin{theorem}\label{thm:SNM-equivalent-BAL-function}
Let $(x,s,\lambda)$ satisfy the affine constraint $\A x = b$.  For any scalars $\mu>0$ and $\rho > 0$, suppose that $(\dx,\ds,\dlam)$ is the search direction generated by the SNM, satisfying \eqref{SNM-Newton-direction-rho}. Then the pair $(\dhx,\ds) := (\B^* \dx,\ds)$ is the unique Newton direction for the minimax problem~\eqref{minimax-eta}, given by 
\begin{equation}\label{reduced-BAL-Newton-update}
    \begin{pmatrix}
    \rho \B^* \H^{-1} \W \B & - \B^* \H^{-1}\W \\
    -  \H^{-1} \W\B & -\rho^{-1} \H^{-1}
    \end{pmatrix}
    \begin{pmatrix}
    \dhx\\
    \ds
    \end{pmatrix}
    = - 
    \begin{pmatrix}
    \B^* (c-y_\rho)\\
    z_\rho - x
    \end{pmatrix}.
\end{equation}
Conversely, suppose that $(\dhx,\ds)$ is the Newton direction for the minimax problem~\eqref{minimax-eta}. Define 
\begin{equation}\label{def:dlam}
    \dlam:= \left({ \A \A^*}\right)^{-1}\A(-s - \ds +c -\A^* \lambda).
\end{equation}
Then $(\dx,\ds,\dlam) = (\B \dhx, \ds,\dlam)$ is the unique search direction given by the SNM. 
\end{theorem}

\begin{proof}
Suppose that $(\dx,\ds,\dlam)$ is the search direction generated by the SNM. 
Let $(\dhx,\ds) = (\B^* \dx, \ds)$. We first show that $ \dx = \B \dhx $. Since $\A \dx = -(\A x -b) = 0$, there exists a unique vector $ \Delta \bar{x} \in \hE$ such that $\dx = \B \Delta \bar{x}$. It follows that 
$$
    \dhx = \B^* \dx = \B^* \B \Delta \bar{x} = \Delta \bar{x}.
$$
Thus, $\dhx = \Delta \bar{x}$ and $ \dx = \B \dhx$. Consequently, 
\begin{equation*}
    \begin{aligned}
         & \rho \B^* \H^{-1} \W \B \dhx - \B^* \H^{-1} \W \ds + \B^* (c-y_\rho)\\  
        =\ &\ \B^* ( \rho \H^{-1} \W \dx - \H^{-1}\W \ds + c -s - \rho (z_\rho-x))\\
        =\ &\ \B^*(c-s - \ds)\\
        =\ &\ \B^* \A^*(\lambda + \dlam)\\
         =\ &\ 0.
    \end{aligned}
\end{equation*}
Here the second and third equalities follow from the third and second equations of \eqref{SNM-Newton-direction-rho}, respectively. The second equation of \eqref{reduced-BAL-Newton-update} follows directly from the third equation of \eqref{SNM-Newton-direction-rho}. Consequently, the linear system~\eqref{reduced-BAL-Newton-update} admits at least one solution. Since the Schur complement of the operator in \eqref{reduced-BAL-Newton-update} relative to $-\rho^{-1} \H^{-1}$ is
$\rho \B^* \W \B \succ 0$, the uniqueness follows directly.

Conversely, suppose that $(\dhx,\ds)$ is the Newton direction for the minimax problem~\eqref{minimax-eta}. Then, 
\begin{equation*}
    \A \dx = \A \B \dhx = 0 = -(\A x -b ),
\end{equation*}
so the first equation of \eqref{SNM-Newton-direction-rho} holds.  Furthermore, the second block row of \eqref{reduced-BAL-Newton-update} gives
\begin{equation*}
-\H^{-1}\W\B\Delta\hx-\rho^{-1}\H^{-1}\Delta s=-(z_\rho-x).
\end{equation*}
Since $\Delta x=\B\Delta\hx$, this is equivalent to
\begin{equation*}
\H^{-1}\W\Delta x+\rho^{-1}\H^{-1}\Delta s=z_\rho-x,
\end{equation*}
which is precisely the third equation of \eqref{SNM-Newton-direction-rho}. It remains to verify the second equation of \eqref{SNM-Newton-direction-rho}. Multiplying the second equation in \eqref{reduced-BAL-Newton-update} on the left by $\rho \B^*$
and adding it to the first equation, we obtain 
\begin{equation*}
    \B^* (c-s-\ds) = 0.
\end{equation*} 
This implies $s + \ds -c \in ({\rm ker}\,\A)^\perp$.  Combined with \eqref{def:dlam}, 
\begin{equation*}
\begin{aligned}
    \A^*(\lambda + \dlam) + s + \ds -c & = (\I_{\E} - \A^* (\A \A^*)^{-1}\A) (s+\ds -c).
\end{aligned}
\end{equation*}
Since $\I_{\E} - \A^* (\A \A^*)^{-1} \A$ is the orthogonal projection onto ${\rm ker}\, \A$, we have 
$$
    \A^*(\lambda + \dlam) + s + \ds -c = 0.
$$
Thus all three equations in \eqref{SNM-Newton-direction-rho} are satisfied. The uniqueness of $(\dx,\ds,\dlam)$ follows from Proposition~\ref{prop:uniqueness-of-SNM-Newton}, and the proof is complete. 
\end{proof}

Theorems~\ref{thm:equivalent-barrier-eta} and~\ref{thm:SNM-equivalent-BAL-function} provide an equivalent characterization of both the parameterized smooth system and the search direction generated by the SNM via the reduced BAL function $\eta_\rho$. This equivalence implies that, in the subsequent algorithmic analysis, it suffices to study the Newton iterations applied to the minimax problem~\eqref{minimax-eta}. This offers a convenient and powerful tool for analyzing the behavior of the SNM. We conclude this section with a summary
of the main characterizations obtained.
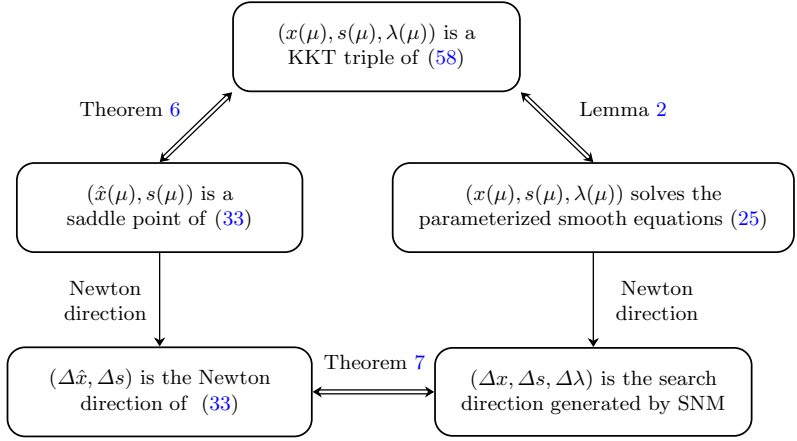
\begin{figure}[H] 
\begin{center}
\resizebox{0.85\linewidth}{!}{%
\begin{tikzpicture}[
    x=1cm,y=1cm,
    font=\normalsize,
    box/.style={
        draw,
        line width=0.8pt,
        rounded corners=8pt,
        minimum height=1.26cm,
        inner xsep=12pt,
        inner ysep=10pt,
        align=center
    },
    darrow/.style={
        double,
        double distance=1.2pt,
        line width=0.6pt,
        -{Stealth[length=4.2pt,width=5.2pt]}
    },
    darrow2/.style={
        double,
        double distance=1.2pt,
        line width=0.6pt,
        {Stealth[length=4.2pt,width=5.2pt]}-{Stealth[length=4.2pt,width=5.2pt]}
    },
    darrow3/.style={
        line width=0.6pt,
        -{Stealth[length=4.2pt,width=5.2pt]}
    }
]

\node[box, minimum width=4.58cm] (TL) at (4.875, 0)
{$(\hx(\mu),s(\mu))$ is a \\ saddle point of \eqref{minimax-eta}};
\node[box, minimum width=3.41cm] (TR) at (11.83, 0)
{$(x(\mu),s(\mu),\lambda(\mu))$ solves the\\ parameterized smooth system~\eqref{SNM-nonlinear-equation}};

\node[box, minimum width=4.60cm] (TM) at ($(TL)!0.5!(TR)+(0,2.60)$)
{$(x(\mu),s(\mu),\lambda(\mu))$ is a \\ KKT triple of \eqref{barrier-problem-without-z}};

\node[box, minimum width=4.88cm] (BL) at (4.875, -2.98)
{$(\dhx,\ds)$ is the Newton\\ direction of ~\eqref{minimax-eta}};
\node[box, minimum width=4.88cm] (BR) at (11.83, -2.98)
{$(\dx,\ds,\dlam)$ is the search \\direction generated by SNM};

\draw[darrow2] (BL.east) -- (BR.west);
\draw[darrow3]  (TL.south) -- (BL.north);
\draw[darrow3]  (TR.south) -- (BR.north);

\draw[darrow2] (TM.south west) -- (TL.north);
\node[font=\normalsize] at ($(TM.south west)!0.55!(TL.north)+(-1.00,0.35)$) {Theorem~\ref{thm:equivalent-barrier-eta}};

\draw[darrow2] (TM.south east) -- (TR.north);
\node[font=\normalsize] at ($(TM.south east)!0.55!(TR.north)+(1.00,0.35)$) {Lemma~\ref{lemma:SNM-equivalent-smooth-KKT}};

\node[align=center] at (4.00,-1.5) {Newton \\ direction};
\node[align=center] at (12.85,-1.5) {Newton\\ direction};
\node at (8.32,-2.5) {Theorem~\ref{thm:SNM-equivalent-BAL-function}};
\end{tikzpicture}%
}
\end{center}
\caption{Summary of equivalence relationships} 
\label{fig:optimality_relation}
\end{figure}

\section{A path-following smoothing Newton method}\label{sec:alg}

This section proposes a path-following smoothing Newton method for symmetric cone programming. The method consists of two phases. In the first phase, an initial point is constructed in a well-defined neighborhood of the central path. The second phase then uses this point to generate iterates that remain in the neighborhood and terminates once the prescribed accuracy is reached.

\subsection{Neighborhood of the central path}
For both practical implementation and theoretical analysis, maintaining iterates within a well-defined neighborhood of the central path is critical. To measure the proximity of a point to the central path and drive the iteration process, we introduce several auxiliary merit functions.

By Theorem~\ref{thm:self-concordant-of-eta}, the function $\et$ is strictly convex in $\hx\in\hE$ and strictly concave in $s\in \E$. Accordingly, we measure its suboptimality by the primal-dual gap function as in \eqref{def:merit-function-for-self-concordant-convex-concave-function}:
\[
\begin{aligned}
\theta_\rho(w;\mu)&=\max_{\tilde{s}}\eta_\rho(\hx,\tilde{s};\mu)-\min_{\tilde{x}}\eta_\rho(\tilde{x},s;\mu)\\
&=\eta_\rho(\hx,{s_\rho(\hx,\mu)};\mu)-\eta_\rho(\hx_\rho (s,\mu),s;\mu),
\end{aligned}
\]
where 
$
s_\rho(\hx,\mu)={{\argmax_{\tilde{s}}}\,\eta_\rho (\hx,\tilde{s};\mu)},\, \hx_\rho (s,\mu)={{\argmin_{\tilde{x}}}\,\eta_\rho(\tilde{x},s;\mu)}.$ 

Let $(\hx(\mu),s(\mu))$ be the saddle point of the minimax problem~\eqref{minimax-eta}, which always exists for any $\mu>0$ and $\rho > 0$ by Remark~\ref{rmk:existence-of-saddle-point}. Then for any $(\hx,s)\in \hE\times \E$, the following inequalities hold:
\begin{equation}\label{eq:max eta>= minimax eta>=min eta}
\begin{aligned}
     \max_{\tilde{s}}\eta_\rho (\hx,\tilde{s};\mu) \ge   \eta_\rho(\hx(\mu),s(\mu);\mu) \ge \min_{\tilde{x}}\eta_\rho (\tilde{x},s;\mu).\\
\end{aligned}
\end{equation}
Let ${\rm val}\, (\mathrm{P}_\mu)$ denote the optimal value of the barrier problem~\eqref{barrier-problem-without-z}. By Theorem~\ref{thm:equivalent-barrier-eta}, 
$$
\eta_\rho(\hx(\mu),s(\mu);\mu) = {\rm val}\, (\mathrm{P}_\mu).
$$
Combining this with \eqref{eq:max eta>= minimax eta>=min eta} yields that, for any $(\hx,s)\in \hE\times \E$,
\begin{equation}\label{error-between-optimal-value-and-eta}
    \left| \et  - {\rm val}\, (\mathrm{P}_\mu) \right| \le \theta_\rho (w;\mu).
\end{equation}
Furthermore, \eqref{eq:max eta>= minimax eta>=min eta} implies that $\theta_\rho (w;\mu)\geq0$, and $\theta_\rho (w;\mu)=0$ holds if and only if $(\hx,s)$ is the saddle point of \eqref{minimax-eta}. This gap therefore provides a certificate of optimality and will be used to measure the proximity to the central path. 

In practice, evaluating the exact primal-dual gap $\theta_\rho (w;\mu)$ is computationally prohibitive, as it requires solving optimization problems. Recall that $x = \bar{x} + \B \hx$. Let $(\dx,\ds,\dlam)$ be the search direction given by \eqref{SNM-Newton-direction-rho}, and let $\Delta w=(\dhx,\ds)$ be the Newton direction for the minimax problem~\eqref{minimax-eta}. For algorithmic purposes, we follow \eqref{def:merit-function-for-self-concordant-convex-concave-function} and introduce easily computable merit functions: 
\begin{subequations}\label{definition delta numerical}
\begin{align}
 	\delta_{\hx,\rho}(w;\mu) &= \sqrt{\dfrac{1}{\mu} \langle \Delta \hx,{ D_{\hx\hx}^2{\et}}\Delta \hx\rangle},\label{definition delta x}\\
 	\delta_{s,\rho}(w;\mu) &= \sqrt{-\dfrac{1}{\mu} \langle \Delta s, D_{ss}^2 \et \Delta s \rangle },\label{definition delta s}\\
	\delta_\rho (w;\mu) &= \sqrt{(\delta_{\hx,\rho}(w;\mu))^2+(\delta_{s,\rho}(w;\mu))^2} = \Vert \Delta w\Vert_{\eta_\rho,w},\label{definition delta}\\
    \xi_{\hx,\rho} (w;\mu)  &= \sqrt{\frac{1}{\mu} \langle \nabla_{\hx} \et, (D^2_{\hx\hx}\et)^{-1} \nabla_{\hx}\et\rangle }, \label{definition xi x}\\
    \xi_{s,\rho} (w;\mu)  &= \sqrt{-\frac{1}{\mu} \langle \nabla_{s} \et, (D^2_{ss}\et)^{-1} \nabla_{s}\et\rangle }, \label{definition xi s}\\
    \xi_\rho (w;\mu) & = \sqrt{  (\xi_{\hx,\rho}(w;\mu))^2 + (\xi_{s,\rho}(w;\mu))^2 } = \| \nabla_{w} \eta_\rho (w;\mu) \|^*_{\eta_\rho,w}.
\end{align}
\end{subequations}
These quantities serve as surrogate measures of the quality of the current iterate. 

\begin{remark}\label{rmk:delta-equivalence}
By Corollary~\ref{coro:differentiability-and-derivatives-of-eta}, we have
$$ 
    \delta_{\hx,\rho}(w;\mu)=\sqrt{\frac{1}{\mu} \langle \Delta x,{ D_{xx}^2 \et}\Delta x\rangle}.
$$
This implies that it is unnecessary to explicitly form $\Delta \hat{x}$ and $\nabla_{\hx\hx}^2 \et$ in practical computations. The quantity $\delta_{\hx,\rho}(w;\mu)$ can be computed directly from $\Delta x$ and $D^2_{x x}\et$.
\end{remark}

If $\xi_\rho (w;\mu) = 0$, then $(\hx,s)$ is the saddle point of the minimax problem \eqref{minimax-eta}. This defines a central path that coincides with the one
generated by IPMs, as shown in Theorem~\ref{thm:equivalent-barrier-eta}.  Specifically,
\begin{equation}
\begin{aligned}
 &\A x = b,\, \A^*\lambda +s = c,\, x\in \interior(\K),\, s\in \interior(\K),\, x \circ s = \mu e\\
 \Longleftrightarrow & \quad
          \A x = b,\, \A^* \lambda + s  = c,\, \xi_\rho(w;\mu)=0.
\end{aligned}
\end{equation}
Motivated by this equivalence, we define the central-path neighborhood for the SNM based on the reduced BAL function $\eta_\rho$ by 
\begin{equation}
      \mathcal{N}(\kappa,\mu,\rho):=
      \left\{(x,s,\lambda)\in\E\times\E\times\mathbb{R}^m\,|\, \A x = b,\, \A^* \lambda + s =c,\, \xi_\rho (w;\mu)\le \kappa \right\},
\end{equation}
where $\kappa = 0.1$ is fixed throughout the algorithm and the complexity analysis.

This neighborhood differs from the standard neighborhoods used in classical interior-point path-following methods~\cite{monteiro1998unified,nesterov1998primal,schmieta2003extension}  or in non-interior path-following methods~\cite{burke2000non,burke1998global,chen2003non,zhao2003globally}.
It is defined via the merit function induced by the
minimax problem, and is tailored to the structure of
the SNM.  The proposed method follows the standard paradigm of path-following methods.  In the first phase, the iterates are driven into $\mathcal{N}(\kappa,\mu^{(0)},\rho)$. In the second phase, the
iterates are maintained in $\mathcal{N}(\kappa,\mu^{(k)},\rho)$ while the smoothing parameter is updated.

\subsection{Two-phase path-following framework}

Before introducing the two-phase framework, we present some additional notation. Let 
$$
v:=(x,s,\lambda), \quad \Delta v:=(\Delta x,\Delta s,\Delta \lambda).
$$
Define $K(\theta_\rho) := \bigl\{w \mid \theta_\rho (w;\mu) < +\infty \bigr\}$.  For a given initial point $w^{(0,0)}:=(\hx^{(0,0)},s^{(0,0)})\in K(\theta_\rho)$, 
let $\bx \in \E$ satisfy $\A \bx =b$, and set 
$$
x^{(0,0)}= \bx + \B\hx^{(0,0)}.
$$
For $t\in (0,1]$, define the perturbed function
\begin{equation}
    \eta_{t,\rho} ( w ;\mu^{(0)} ) : =  \eta_\rho ( w ;\mu^{(0)}) - t \langle \nabla_{w} \eta_\rho (w^{(0,0)};\mu^{(0)}), w \rangle.
\end{equation}
The purpose of this perturbation is to construct a continuation path from the initial point to a point satisfying the neighborhood condition for the fixed parameter $\mu^{(0)}$. 
Since $\eta_\rho$ is a nondegenerate $\mu$-self-concordant convex-concave function, $\eta_{t,\rho}$ inherits the same property for any $t\ge 0$.

The first phase of PFSNM aims to generate a feasible point in the neighborhood $\mathcal{N}(\kappa,\mu^{(0)},\rho)$. The parameter $t$ scales the perturbation term so that the initial point $(x^{(0,0)},s^{(0,0)},\lambda^{(0,0)})$ lies in the central-path neighborhood of the perturbed problem. For this purpose, consider the minimax problem
$$
\min_{\hx \in \hE} \max_{s \in \E}\left\{ \eta_{t,\rho} (w;\mu^{(0)}) \right\}
$$ 
with the first-order optimality conditions 
\begin{equation}\label{alg:eta_t minimax optimality condition}
\begin{aligned}
     & \B^* (c - y_\rho) - t \B^* (c - y_\rho^{(0,0)}) = 0,\\
     & z_\rho - x - t (z_\rho^{(0,0)} - x^{(0,0)}) = 0,
\end{aligned}
\end{equation}
where $y_\rho^{(0,0)} := y_\rho (w^{(0,0)};\mu^{(0)})$ and $z_\rho^{(0,0)} := z_\rho (w^{(0,0)};\mu^{(0)})$.

Applying Newton's method to the nonlinear system \eqref{alg:eta_t minimax optimality condition} leads to the search direction $(\Delta \hx, \Delta s)$ satisfying 
\begin{equation}\label{eq:etat-Newton-direction}
        \begin{pmatrix}
    \rho \B^* \H^{-1} \W \B & - \B^* \H^{-1}\W \\
    -  \H^{-1} \W\B & -\rho^{-1} \H^{-1}
    \end{pmatrix}
    \begin{pmatrix}
    \dhx\\
    \ds
    \end{pmatrix}
    = - 
    \begin{pmatrix}
    \B^* (c-y_\rho)\\
    z_\rho - x
    \end{pmatrix}
    + t 
    \begin{pmatrix}
    \B^* (c-y_\rho^{(0,0)})\\
    z_\rho^{(0,0)} - x^{(0,0)}
    \end{pmatrix}
    .
\end{equation}
In practical computations, the variable $\hx$ is not formed explicitly. Following Theorem~\ref{thm:SNM-equivalent-BAL-function}, we instead
solve the system in $(x,s,\lambda)$:
\begin{equation}\label{eq:SNMt-Newton-direction}
  \begingroup 
    \fontsize{9pt}{12pt}\selectfont
        \begin{pmatrix}
    \A & 0 & 0\\
    0 & \I_{\E} & \A^*\\
    \H^{-1}\W & \rho^{-1} \H^{-1} & 0
    \end{pmatrix}
    \begin{pmatrix}
    \dx\\
    \ds\\
    \dlam
    \end{pmatrix}
    = -
    \begin{pmatrix}
    \A x -b \\
    \A^* \lambda + s -c\\
    x - z_\rho
    \end{pmatrix}
    + t 
    \begin{pmatrix}
    0\\
    \A^* \lambda^{(0,0)} + s^{(0,0)} -c\\
    x^{(0,0)}-z_\rho^{(0,0)}
    \end{pmatrix}.
    \endgroup
\end{equation}
The resulting search direction coincides with the one obtained by solving \eqref{eq:etat-Newton-direction} under the constraint $\A x=b$.
This equivalence is formalized in the following corollary.
\begin{corollary}\label{coro:Newton-direction-t-equivalence}
    Let $(x,s,\lambda)$ satisfy the affine constraint $\A x = b$. For the given point $(x^{(0,0)},s^{(0,0)},\lambda^{(0,0)})$ and any $\mu^{(0)}>0,\,\rho > 0$, suppose that $(\dx,\ds,\dlam)$ solves \eqref{eq:SNMt-Newton-direction}. Then the pair $(\dhx,\ds) := (\B^* \dx,\ds)$ is the unique solution of \eqref{eq:etat-Newton-direction}.
Conversely, if $(\dhx,\ds)$ solves \eqref{eq:etat-Newton-direction}, define 
\begin{equation}
    \dlam:= \left({ \A \A^*}\right)^{-1}\A(-s - \ds +c -\A^* \lambda + t (\A^* \lambda^{(0,0)} + s^{(0,0)} - c)).
\end{equation}
Then $(\dx,\ds,\dlam): = (\B \dhx, \ds,\dlam)$ is the unique solution of \eqref{eq:SNMt-Newton-direction}. 
\end{corollary}
\begin{proof}
 The proof follows the same arguments as in Theorem~\ref{thm:SNM-equivalent-BAL-function}, with a minor modification
due to the shift term. Details are omitted.
\end{proof}

We are now ready to state the first phase of the proposed method. Starting from an arbitrary point $w^{(0,0)}\in K(\theta_\rho)$, the purpose of this phase is to follow the perturbed minimax problems defined by $\eta_{t,\rho}$ and to drive the iterate into the neighborhood $\mathcal N(\kappa,\mu^{(0)},\rho)$ for the fixed smoothing parameter $\mu^{(0)}$. The parameter $t$ plays the role of a homotopy parameter: for a suitable $t^{(0)}\in(0,1]$ such that
\begin{equation}\label{ini-assumption}
        (1-t^{(0)}) \delta_\rho (w^{(0,0)}; \mu^{(0)}) \le \frac{\kappa}{2},
\end{equation}
the initial point is close to the saddle point of the perturbed problem, while decreasing $t$ gradually removes the perturbation and recovers the original reduced minimax problem.

\begin{algorithm}[H]
\caption{The first phase of PFSNM}
\label{alg:initiation}
\begin{algorithmic}[1]
\Require $w^{(0,0)}\in K(\theta_\rho)$, $\lambda^{(0,0)}\in \R^m$, $\rho>0$, $\mu^{(0)}>0$, and $t^{(0)}\in(0,1]$ satisfying \eqref{ini-assumption}.

\For{$j=0,1,2,\ldots$}
    \If{$\delta_\rho(w^{(0,j)};\mu^{(0)})\leq\kappa$}
        \State Compute the Newton direction $\Delta v^{(0,j)}$ from \eqref{SNM-Newton-direction-rho}. 
        \State \textbf{return} $v^{(0)} = v^{(0,j)}+\Delta v^{(0,j)}.$
    \Else
    \State Update
    \begin{equation}\label{alpha}
    \;\;\;\quad\alpha^{(j)}
    =
    \min\left\{
    \frac{\kappa}{
    4t^{(j)}
    \left\|
    \bigl(D^2_{ww}\eta_\rho(w^{(0,j)};\mu^{(0)})\bigr)^{-1}
    \nabla_w\eta_\rho(w^{(0,0)};\mu^{(0)})
    \right\|_{\eta_\rho,w^{(0,j)}}},
    1
    \right\}.
    \end{equation}
    \State Set
    $
    t^{(j+1)}=(1-\alpha^{(j)})t^{(j)}.
    $
    \State Compute the search direction $\Delta v^{(0,j)}$ from \eqref{eq:SNMt-Newton-direction} with $t=t^{(j+1)}$ and set
    \begin{equation*}
    v^{(0,j+1)} = v^{(0,j)}+\Delta v^{(0,j)}.
    \end{equation*}
    \EndIf

\EndFor

\end{algorithmic}
\end{algorithm}

Throughout Algorithm~\ref{alg:initiation},
the iterates $(x^{(0,j)}, s^{(0,j)}, \lambda^{(0,j)})$ always satisfy the primal constraint, as follows from \eqref{eq:SNMt-Newton-direction}. When Algorithm~\ref{alg:initiation} terminates, the second equation of \eqref{SNM-Newton-direction-rho} yields 
$$
    \A^* \lambda^{(0)} + s^{(0)} - c = 0.
$$
Furthermore, since $\delta_\rho(w^{(0,j)};\mu^{(0)})\le \kappa=0.1<1$, 
Theorem~\ref{thm:properties-for-self-concordant-convex-concave-function}(iii) yields
\begin{equation*}
\xi_\rho(w^{(0)};\mu^{(0)})
\le
\left(
\frac{\delta_\rho(w^{(0,j)};\mu^{(0)})}
     {1-\delta_\rho(w^{(0,j)};\mu^{(0)})}
\right)^2 \le
\frac12\delta_\rho(w^{(0,j)};\mu^{(0)})
\le \kappa.
\end{equation*}
Hence $v^{(0)} \in \mathcal{N}(\kappa, \mu^{(0)},\rho)$.

Once a point in $\mathcal N(\kappa,\mu^{(0)},\rho)$ has been obtained, the second phase decreases the smoothing parameter and recenters the iterate for each new value of $\mu$.  The analysis in Section~\ref{sec5} will show that, after each update of $\mu$, only one Newton step is sufficient to return to the neighborhood.

\begin{algorithm}[H]
\caption{The second phase of PFSNM}
\label{alg:main}
\begin{algorithmic}[1]
\Require $v^{(0)}\in \mathcal{N}(\kappa,\mu^{(0)},\rho)$, $\rho>0$, $\mu^{(0)}>0$, $\varepsilon>0$, and $\sigma\in(0,1)$.
\For{$k=0,1,2,\ldots$}
    \If{$\mu^{(k)}\le \varepsilon$}
        \State \textbf{return} $v^{(k)}$.
    \Else
        \State Update $\mu^{(k+1)}=\sigma\mu^{(k)}$.
        \State Compute the search direction $\Delta v^{(k)}$ from \eqref{SNM-Newton-direction-rho} with $\mu=\mu^{(k+1)}$, and set
        \begin{equation*}
        v^{(k+1)}=v^{(k)}+\Delta v^{(k)}.
        \end{equation*}   
    \EndIf
\EndFor
\end{algorithmic}
\end{algorithm}

After each reduction of the smoothing parameter, a single full Newton step is sufficient to recenter the iterate; more precisely, the resulting point $v^{(k+1)}$ belongs to the new neighborhood $\mathcal{N}(\kappa,\mu^{(k+1)},\rho)$. This fact is established in Theorem~\ref{thm:complexity of phase 2}. Therefore, Algorithm~\ref{alg:main} retains the standard path-following structure.

Several remarks on Algorithms~\ref{alg:initiation} and \ref{alg:main}
are in order:
\begin{enumerate}
    \item In practical computations, the variable $\hx$ is never formed explicitly. Computations are performed directly with the primal variable $x$ since $\delta_\rho(w^{(0,j)};\mu^{(0)})$ can be evaluated from $(x,s,\lambda)$; see   Remark~\ref{rmk:delta-equivalence} for details.
    \item There is no need to form the Newton systems of $\eta_\rho$ or $\eta_{t,\rho}$ with respect to $(\hx,s)$ explicitly. By Theorem~\ref{thm:SNM-equivalent-BAL-function} and Corollary~\ref{coro:Newton-direction-t-equivalence}, updating $(x,s,\lambda)$ is in fact equivalent to updating $(\hx,s)$.
\item The iterate entering Algorithm~\ref{alg:main} satisfies the primal and dual constraints. Consequently, by \eqref{SNM-Newton-direction-rho} and the iteration scheme, all subsequent iterates remain feasible:
\begin{equation}\label{eq:initial=0}
\begin{aligned}
     \A x^{(k)} &=b,\, &\A^* \lambda^{(k)} + s^{(k)} =c,\quad \forall\,  k\ge 0.
    \end{aligned}
\end{equation}

\item  Once Algorithm~\ref{alg:main} terminates, we have 
\begin{equation*}
\mu^{(k)} \le \varepsilon \;\text{ and } \;\xi_\rho(w^{(k)};\mu^{(k)})=\|\nabla_{w} \eta_\rho (w^{(k)};\mu^{(k)})\|_{\eta_\rho,w^{(k)},\mu^{(k)}}^*\le \kappa.
\end{equation*}
It follows from the definition of $\xi_\rho$ and Corollary~\ref{coro:differentiability-and-derivatives-of-eta} that
\begin{equation*}
\begin{aligned}
      \|z^{(k)}_\rho-x^{(k)}\| & \le \sqrt{\langle z^{(k)}_\rho -x^{(k)}, \H (z^{(k)}_\rho-x^{(k)}) \rangle}\\
    &  \le \sqrt{\frac{\mu^{(k)}}{\rho}} \|  \nabla_{w} \eta_\rho ( w^{(k)};\mu^{(k)})\|^*_{\eta_\rho,w^{(k)},\mu^{(k)}}\\
    & \le \sqrt{\frac{\varepsilon}{\rho}} \kappa.
\end{aligned}
\end{equation*}
\end{enumerate}

Recall that $\Phi_\rho (x^{(k)},s^{(k)};\mu^{(k)})  = 2(x^{(k)}-z_\rho( x^{(k)},s^{(k)};\mu^{(k)} ))$. By (iii) and (iv), when Algorithm~\ref{alg:main} terminates, an approximate KKT triple $(x^{(k)}, s^{(k)}, \lambda^{(k)})$ for the original SCP problem~\eqref{SCP} is obtained, satisfying 
\begin{equation}
 \A x^{(k)} =b,\; \A^* \lambda^{(k)} + s^{(k)} = c,\; \Vert \Phi_\rho (x^{(k)},s^{(k)};\mu^{(k)}) \Vert \le  2\kappa\sqrt{ \frac{\varepsilon}{\rho} } .
\end{equation}

\subsection{Linear systems in the algorithms}\label{subsec:linear system}

In practical computations, the primal constraint is satisfied at all iterates. Thus, the Newton systems \eqref{eq:SNMt-Newton-direction}  and \eqref{SNM-Newton-direction-rho}  arising in Algorithms~\ref{alg:initiation} and \ref{alg:main} can be written in the unified form

\begin{equation}\label{Newton direction}
\begin{pmatrix}
 \A&  0 & 0\\
0 &\I_{\E} &\A^*\\
\H^{-1}\W&\rho^{-1}\H^{-1}&0
\end{pmatrix}
\begin{pmatrix}
\Delta x \\ \Delta s \\ \Delta \lambda
\end{pmatrix}=
\begin{pmatrix}
0\\ r_1 \\ r_2
\end{pmatrix},
\end{equation}
where 
$$
\begin{aligned}
    & r_1 = 
    \begin{cases}
        -\A^* \lambda -s + c + t( \A^* \lambda^{(0,0)} + s^{(0,0)}-c ), & \text{for the first phase system } \eqref{eq:SNMt-Newton-direction};\\
        0, & \text{for the second phase system } \eqref{SNM-Newton-direction-rho},
    \end{cases}
    \\
     & r_2 = \begin{cases}
     z_\rho-x-t(z_\rho^{(0,0)}-x^{(0,0)}), & \text{for the first phase system } \eqref{eq:SNMt-Newton-direction};\\
     z_\rho-x, & \text{for the second phase system } \eqref{SNM-Newton-direction-rho}.
     \end{cases}
\end{aligned}
$$

Directly solving the full Newton system \eqref{Newton direction} is computationally expensive, particularly for large-scale problems when the dimensions of the cone variables $x$ and $s$ are significantly larger than the number of constraints. To reduce the computational cost, we eliminate $\dx$ and $\ds$ from \eqref{Newton direction},  which leads to the following Schur-complement system:
\begin{subequations}\label{Schur complement system}
\begin{align}
\A \W^{-1} \A^* \dlam &= \A \W^{-1} r_1 -\rho\A \W^{-1}\H r_2,\\
\Delta s &= r_1 -\A^*\dlam,\\
\Delta x&= 
 \W^{-1}(\H r_2 -\rho^{-1}\Delta s).
\end{align}
\end{subequations}

For both IPMs and classical SNMs, the dominant computational cost typically comes from forming and factorizing the Schur complement $\A \D \A^{*}$. In our method, $\D=\W^{-1}$, where $\W$ is an iterate-dependent operator determined by the barrier function in PFSNM. Accordingly, improving the efficiency of forming $\A \D \A^{*}$ is crucial for large-scale computation. To this end, Proposition~\ref{pro:closed_forms} provides closed-form expressions for  $\A \D \A^{*}$ in PFSNM for the three most common symmetric cones.

\begin{proposition}\label{pro:closed_forms}
	Let $\K$ be one of the symmetric cones considered below, and let $e$ denote the corresponding Jordan identity element. Then the corresponding Schur complement admits
 the following explicit representations.
	\begin{enumerate}
		\item[(i)] Let $\K=\R^{n}_{+}$. Then, for any $z\in\R^{n}_{++}$, $\phi(z)=-\sum_{i=1}^{n}\ln\, z_i$, and
		\begin{equation*}
		\A \W^{-1}\A^*=\frac{\rho}{\mu}\A \big(\Diag(z_\rho)\big)^{2} \A^*.
		\end{equation*}
		\item[(ii)] Let $\K=\Q^{n+1}$.
		Then, for any $z\in\interior(\Q^{n+1})$,
		$
		\phi(z)=-\ln\big(z_{0}^{2}-\|\bar z\|^{2}\big),
        $ and
		\begin{equation*}
		\A \W^{-1}\A^*=\frac{\rho}{\mu}\Big(\det(z_\rho)\,\A\A^{*}+2(\A z_\rho)(\A z_\rho)^{*}-2\det(z_\rho)\,(\A e)(\A e)^{*}\Big).
		\end{equation*}
		\item[(iii)] Let $\K=\S^{n}_{+}$. Then, for any $Z\in\S^{n}_{++}$, $\phi(Z)=-\ln\det(Z)$, and 
	\begin{equation*}
			\A \W^{-1}\A^*=\frac{\rho}{\mu}\,\A (Z_\rho\otimes_{s}Z_\rho)\A^{*}.
		\end{equation*}
	\end{enumerate}
\end{proposition}
\begin{proof}
By definition, ${\W}=\dfrac{\mu}{\rho}D^2\phi(z_\rho)$. The results follow directly from the explicit formulas (see \cite[Proposition 2.6.1]{vieira2007jordan}) for $D^2\phi$.
\end{proof}

The importance of explicit Schur-complement formulas has already been noted in the SDP literature. The SNM in \cite{chen2003non}, based on the smoothing FB function, has a Schur-complement formation cost comparable to that of the most expensive Alizadeh--Haeberly--Overton (AHO) direction in IPMs. If the smoothing CHKS function is used instead, the formation cost becomes lower than that of AHO, but remains higher than that of the Nesterov--Todd (NT) direction. In our earlier work \cite{zhang2024iprsdp}, we showed that once an explicit formula for the Schur complement is available, the formation cost can be reduced to the same order as that of the NT direction. Compared with~\cite{zhang2024iprsdp}, this paper leverages self-concordant properties to derive this explicit Schur-complement formula in a simpler and more direct way.

For SOCP, the advantage of the explicit Schur complement is even more significant. As shown in case (ii) of Proposition~\ref{pro:closed_forms}, the Schur complement takes the form of a scaled matrix plus two rank-one updates. However, the vector $u:=\A z_\rho$ appearing in these rank-one terms is typically dense, which causes the resulting Schur complement to be dense as well. Figure~\ref{fig:SCM of SNM} illustrates this structure and shows how these components combine to yield a fully dense matrix. This density poses a major computational challenge in large-scale settings. Even accelerating the evaluation of $\D$ as described in \cite{fukushima2002smoothing} does not solve the problem, because the Schur complement remains dense. The explicit representation in Proposition~\ref{pro:closed_forms} avoids forming this dense matrix. The terms $\A\A^*$ and $\A e$ depend only on the problem data and can be precomputed once. Each subsequent iteration then requires only one matrix-vector product $u=\A z_\rho$, along with simple low-rank updates and the scalar $\det(z_\rho)$. With this structure, one can apply either the product-form Cholesky factorization approach in \cite{alizadeh2003second} or the expanded sparse representation technique in \cite{zhang2026iprsocp}. Both approaches exploit the low-rank structure and avoid forming a dense Schur complement.

\vspace{2em}

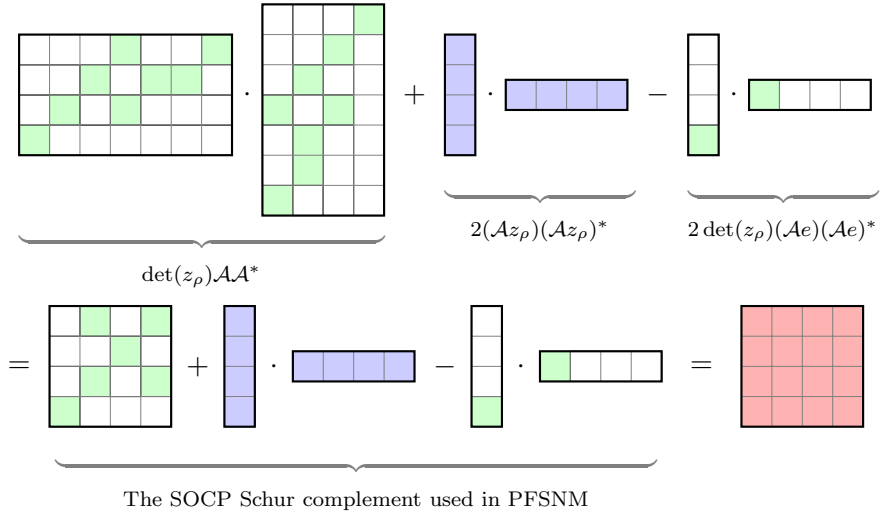
\begin{figure}
\begin{tikzpicture}

    \filldraw[fill=green!20!white, draw=green!40!black] (0,0.4) rectangle (2.8,2.0);

    \filldraw[fill=white] (0.4,0.4) rectangle (2.8,0.8);

    \filldraw[fill=white] (0,0.8) rectangle (0.4,1.2);
    \filldraw[fill=white] (0.8,0.8) rectangle (1.2,1.2);
    \filldraw[fill=white] (1.6,0.8) rectangle (2.8,1.2);

    \filldraw[fill=white] (0,1.2) rectangle (0.8,1.6);
    \filldraw[fill=white] (1.2,1.2) rectangle (1.6,1.6);
    \filldraw[fill=white] (2.4,1.2) rectangle (2.8,1.6);

    \filldraw[fill=white] (0,1.6) rectangle (1.2,2.0);
    \filldraw[fill=white] (1.6,1.6) rectangle (2.4,2.0);

    \draw[step=0.4, very thin, color=gray] (0,0.4) grid (2.8,2.0);
    \draw[thick] (0,0.4) rectangle (2.8,2.0);

    \node at (3.0,1.2) {\large $\cdot$};

    \filldraw[fill=green!20!white, draw=green!40!black] (3.2,-0.4) rectangle (4.8,2.4);

    \filldraw[fill=white] (3.2,0.0) rectangle (3.6,0.8);
    \filldraw[fill=white] (3.2,1.2) rectangle (3.6,2.4);

    \filldraw[fill=white] (3.6,-0.4) rectangle (4.0,0.0);
    \filldraw[fill=white] (3.6,0.8) rectangle (4.0,1.2);
    \filldraw[fill=white] (3.6,1.6) rectangle (4.0,2.4);

    \filldraw[fill=white] (4.0,-0.4) rectangle (4.4,0.8);
    \filldraw[fill=white] (4.0,1.2) rectangle (4.4,1.6);
    \filldraw[fill=white] (4.0,2.0) rectangle (4.4,2.4);

    \filldraw[fill=white] (4.4,-0.4) rectangle (4.8,2.0);

    \draw[step=0.4, very thin, color=gray] (3.2,-0.4) grid (4.8,2.4);
    \draw[thick] (3.2,-0.4) rectangle (4.8,2.4);

    \node at (5.2,1.2) {\large $+$};

    \filldraw[fill=blue!20!white, draw=blue!40!black] (5.6,0.4) rectangle (6.0,2.0);
    \foreach \y in {0.8,1.2,1.6}
        \draw[very thin, gray] (5.6,\y) -- (6.0,\y);
    \draw[thick] (5.6,0.4) rectangle (6.0,2.0);

    \node at (6.20,1.2) {\large $\cdot$};

    \filldraw[fill=blue!20!white, draw=blue!40!black] (6.4,1.0) rectangle (8.0,1.4);
    \foreach \x in {6.8,7.2,7.6}
        \draw[very thin, gray] (\x,1.0) -- (\x,1.4);
    \draw[thick] (6.4,1.0) rectangle (8.0,1.4);

    \node at (8.4,1.2) {\large $-$};

    \filldraw[fill=green!20!white, draw=green!40!black] (8.8,0.4) rectangle (9.2,2.0);
    \filldraw[fill=white] (8.8,0.8) rectangle (9.2,2.0);
    \foreach \y in {0.8,1.2,1.6}
        \draw[very thin, gray] (8.8,\y) -- (9.2,\y);
    \draw[thick] (8.8,0.4) rectangle (9.2,2.0);

    \node at (9.4,1.2) {\large $\cdot$};

    \filldraw[fill=green!20!white, draw=green!40!black] (9.6,1.0) rectangle (11.2,1.4);
    \filldraw[fill=white] (10.0,1.0) rectangle (11.2,1.4);
    \foreach \x in {10.0,10.4,10.8}
    \draw[very thin, gray] (\x,1.0) -- (\x,1.4);
    \draw[thick] (9.6,1.0) rectangle (11.2,1.4);

    \draw (2.4,-0.75) node[rotate=0]
    {\color{gray}$\underbrace{\hspace{4.8cm}}$};
    \node at (2.4,-1.20) {\small$\det(z_\rho)\A\A^*$};

    \draw (6.85,-0.15) node[rotate=0]
    {\color{gray}$\underbrace{\hspace{2.5cm}}$};
    \node at (6.85,-0.60) {\small$2(\A z_\rho)(\A z_\rho)^*$};

    \draw (10.05,-0.15) node[rotate=0]
    {\color{gray}$\underbrace{\hspace{2.5cm}}$};
    \node at (10.05,-0.60) {\small$2\det(z_\rho)(\A e)(\A e)^*$};

    \node at (0,-2.4) {\large $=$};

    \filldraw[fill=green!20!white, draw=green!40!black] (0.4,-3.2) rectangle (2.0,-1.6);

    \filldraw[fill=white] (0.8,-3.2) rectangle (2.0,-2.8);
    \filldraw[fill=white] (0.4,-2.8) rectangle (0.8,-2.4);
    \filldraw[fill=white] (1.2,-2.8) rectangle (1.6,-2.4);
    \filldraw[fill=white] (0.4,-2.4) rectangle (1.2,-2.0);
    \filldraw[fill=white] (1.6,-2.4) rectangle (2.0,-2.0);
    \filldraw[fill=white] (0.4,-2.0) rectangle (0.8,-1.6);
    \filldraw[fill=white] (1.2,-2.0) rectangle (1.6,-1.6);

    \draw[step=0.4, very thin, color=gray] (0.4,-3.2) grid (2.0,-1.6);
    \draw[thick] (0.4,-3.2) rectangle (2.0,-1.6);

    \node at (2.35,-2.4) {\large $+$};

    \filldraw[fill=blue!20!white, draw=blue!40!black] (2.7,-3.2) rectangle (3.1,-1.6);
    \foreach \y in {-2.8,-2.4,-2.0}
        \draw[very thin, gray] (2.7,\y) -- (3.1,\y);
    \draw[thick] (2.7,-3.2) rectangle (3.1,-1.6);

    \node at (3.35,-2.4) {\large $\cdot$};

    \filldraw[fill=blue!20!white, draw=blue!40!black] (3.6,-2.6) rectangle (5.2,-2.2);
    \foreach \x in {4.0,4.4,4.8}
        \draw[very thin, gray] (\x,-2.6) -- (\x,-2.2);
    \draw[thick] (3.6,-2.6) rectangle (5.2,-2.2);

    \node at (5.6,-2.4) {\large $-$};

    \filldraw[fill=green!20!white, draw=green!40!black] (5.95,-3.2) rectangle (6.35,-1.6);
    \filldraw[fill=white] (5.95,-2.8) rectangle (6.35,-1.6);
    \foreach \y in {-2.8,-2.4,-2.0}
        \draw[very thin, gray] (5.95,\y) -- (6.35,\y);
    \draw[thick] (5.95,-3.2) rectangle (6.35,-1.6);

    \node at (6.6,-2.4) {\large $\cdot$};

    \filldraw[fill=green!20!white, draw=green!40!black] (6.85,-2.6) rectangle (8.45,-2.2);
    \filldraw[fill=white] (7.25,-2.6) rectangle (8.45,-2.2);
    \foreach \x in {7.25,7.65,8.05}
        \draw[very thin, gray] (\x,-2.6) -- (\x,-2.2);
    \draw[thick] (6.85,-2.6) rectangle (8.45,-2.2);

    \node at (8.975,-2.4) {\large $=$};

    \begin{scope}[shift={(9.5,-3.2)}]
        \foreach \i in {0,1,2,3}{
            \foreach \j in {0,1,2,3}{
                \fill[red!30!white] ({0.4*\j},{0.4*\i}) rectangle ({0.4*(\j+1)},{0.4*(\i+1)});
            }
        }
        \draw[step=0.4, very thin, color=gray] (0,0) grid (1.6,1.6);
        \draw[thick] (0,0) rectangle (1.6,1.6);
    \end{scope}

    \draw (4.43,-3.72) node[rotate=0]
    {\color{gray}$\underbrace{\hspace{7.9cm}}$};
    \node at (4.43,-4.18) {\small The SOCP Schur complement used in PFSNM};

\end{tikzpicture}
\caption{Structure of the SOCP Schur complement in PFSNM}
 \label{fig:SCM of SNM}
\end{figure}

\section{Complexity analysis}\label{sec5}

In this section, we establish a polynomial iteration complexity for PFSNM applied to SCP. The analysis consists of two parts. We first derive an iteration complexity for the first phase. Starting from the point produced by this phase, Algorithm~\ref{alg:main} then attains an iteration complexity of order $\O(\sqrt{\nu}\ln(1/\varepsilon))$, matching the classical short-step interior-point complexity \cite{vavasis1996primal}.

The iteration complexity for the first phase is presented in the following theorem.

\begin{theorem}\label{initialization-complexity-analysis}
    Suppose that $w^{(0,0)}\in K(\theta_\rho)$ and that $t^{(0)}$ in Algorithm~\ref{alg:initiation} satisfies \eqref{ini-assumption}.
    Then Algorithm~\ref{alg:initiation} requires at most 
    \begin{equation}
       \O\left(
1+
\frac{
4t^{(0)}
}{\kappa}\Theta_1\left(\dfrac{\theta_\rho(w^{(0,0)};\mu^{(0)})}{\mu^{(0)}}\right)
\right)
    \end{equation}
    iterations to obtain a starting point $v^{(0)} \in \mathcal{N}(\kappa,\mu^{(0)},\rho)$, where $\Theta_1 (\cdot)$ is a properly chosen universal positive, continuous, and nondecreasing function on $\R_+$.
\end{theorem}

\begin{proof}
 Since $\mu$ remains fixed in the first phase, we write $\eta_\rho(w)$ and $\eta_{t,\rho} (w)$ in place of $\eta_\rho(w;\mu^{(0)})$ and $\eta_{t,\rho} (w;\mu^{(0)})$, respectively. Define the merit function associated with $\eta_{t,\rho}$ by
\begin{equation*}
    \delta_{t,\rho} (w): = \Vert (D^2_{w w} \eta_\rho (w))^{-1} \nabla_{w} \eta_{t,\rho} (w)\Vert_{\eta_{t,\rho},w} = \Vert (D^2_{w w} \eta_\rho (w))^{-1} \nabla_{w} \eta_{t,\rho} (w)\Vert_{\eta_\rho,w}.
\end{equation*}
Noting that $\nabla_{w} \eta_{t,\rho} (w) = \nabla_{w} \eta_\rho (w) - t \nabla_{w} \eta_\rho(w^{(0,0)})$,
\begin{equation}\label{eq:delta-t0-delta-t}
    \delta_{t^{(0)},\rho}(w^{(0,0)}) = (1-t^{(0)})\delta_\rho (w^{(0,0)}).
\end{equation}
We now prove by induction that $\delta_{t^{(j)},\rho} (w^{(0,j)}) \le \frac{\kappa}{2}$ at each iteration $j$. For $j=0$, the claim follows from \eqref{ini-assumption} and \eqref{eq:delta-t0-delta-t}. Assume that $\delta_{t^{(j)},\rho} (w^{(0,j)}) \le \frac{\kappa}{2}$ at the $j$-th iteration. Then 
\begin{equation*}
 \begin{aligned}
  &\delta_{t^{(j+1)},\rho} (w^{(0,j)}) -  \delta_{t^{(j)},\rho} (w^{(0,j)}) \\
  = &\ \Vert (D^2_{w w} \eta_\rho (w^{(0,j)}))^{-1} \nabla_{w} \eta_{t^{(j+1)},\rho} (w^{(0,j)})\Vert_{\eta_{\rho},w^{(0,j)}}\\
  & \qquad \qquad \qquad - \Vert (D^2_{w w} \eta_\rho (w^{(0,j)}))^{-1} \nabla_{w} \eta_{t^{(j)},\rho} (w^{(0,j)})\Vert_{\eta_\rho,w^{(0,j)}} \\
  \le &\  \alpha^{(j)}t^{(j)} \Vert (D^2_{w w} \eta_\rho (w^{(0,j)}))^{-1} \nabla_{w} \eta_\rho(w^{(0,0)}) \Vert_{\eta_\rho, w^{(0,j)}}\\
  \le &\ \frac{\kappa}{4}.
 \end{aligned}
 \end{equation*}
Here the first inequality follows from the identity for $\nabla_{w} \eta_{t,\rho} (w) $ and from the fact that $\| \cdot \|_{\eta_\rho,w^{(0,j)}}$ defines a norm. The second inequality follows from~\eqref{alpha}.   Thus, 
 \begin{equation}\label{inq-j+1}
     \delta_{t^{(j+1)},\rho} (w^{(0,j)}) \le \frac{\kappa}{2} + \frac{\kappa}{4} = \frac{3 \kappa}{4} < \kappa. 
  \end{equation}
By \eqref{inq-j+1} and Theorem~\ref{thm:properties-for-self-concordant-convex-concave-function}(ii)--(iii), we have 
\begin{equation*}
    \delta_{t^{(j+1)},\rho}(w^{(0,j+1)}) \le \left( \frac{\delta_{t^{(j+1)},\rho}(w^{(0,j)})}{1-\delta_{t^{(j+1)},\rho}(w^{(0,j)})} \right)^2 \le \frac{\kappa}{2},
\end{equation*}
which completes the induction argument. Hence, 
\begin{equation}\label{eq:phase 1 delta}
\delta_{t^{(j)},\rho}(w^{(0,j)}) \le \frac{\kappa}{2}, \quad \forall\, j\ge 0.
\end{equation}

By applying Nemirovski's estimate \cite[Lemma 8.3(a)]{nemirovski1999self} to the standard self-concordant convex-concave function $\frac{\eta_\rho(w;\mu^{(0)})}{\mu^{(0)}}$,
there exists a universal positive, continuous, and nondecreasing function
$\Theta_1$ such that
\begin{equation}\label{ini-proof-4}
    \Vert (D^2_{w w} \eta_\rho(w^{(0,j)}))^{-1} \nabla_{w} \eta_\rho(w^{(0,0)}) \Vert_{\eta_\rho, w^{(0,j)}} \le \Theta_1\left(\frac{\theta_\rho(w^{(0,0)};\mu^{(0)})}{\mu^{(0)}}\right).
\end{equation}

It remains to estimate the number of updates of $t$. If
$\alpha^{(j)}=1$, then $t^{(j+1)}=(1-\alpha^{(j)})t^{(j)}=0$. Hence, by \eqref{eq:phase 1 delta},
\begin{equation*}
\delta_\rho(w^{(0,j+1)};\mu^{(0)})
=
\delta_{t^{(j+1)},\rho}(w^{(0,j+1)})
\le
\frac{\kappa}{2}.
\end{equation*}
Thus Algorithm~\ref{alg:initiation} terminates after at most one additional Newton step.

We next consider the case where $\alpha^{(j)}<1$. Using the definition of
$\alpha^{(j)}$ in \eqref{alpha}, we have
\begin{equation*}
\alpha^{(j)}t^{(j)}
=
\frac{\kappa}
{4
\left\|
(D^2_{ww}\eta_\rho(w^{(0,j)};\mu^{(0)}))^{-1}
\nabla_w\eta_\rho(w^{(0,0)};\mu^{(0)})
\right\|_{\eta_\rho,w^{(0,j)}}}.
\end{equation*}
Therefore,
\begin{equation*}
\begin{aligned}
t^{(j+1)}
&=(1-\alpha^{(j)})t^{(j)}  \\
&=t^{(j)}
-
\frac{\kappa}
{4
\left\|
(D^2_{ww}\eta_\rho(w^{(0,j)};\mu^{(0)}))^{-1}
\nabla_w\eta_\rho(w^{(0,0)};\mu^{(0)})
\right\|_{\eta_\rho,w^{(0,j)}}}.
\end{aligned}
\end{equation*}
By \eqref{ini-proof-4}, this implies
\begin{equation*}
t^{(j+1)}
\le
t^{(j)}
-
\frac{\kappa}
{4\Theta_1\left( \frac{\theta_\rho(w^{(0,0)};\mu^{(0)} )} {\mu^{(0)}} \right)}.
\end{equation*}
Consequently, as long as Algorithm~\ref{alg:initiation} has not yet reached a point satisfying
\begin{equation*}
t^{(j)}\Theta_1\left(\frac{\theta_\rho(w^{(0,0)};\mu^{(0)})}{\mu^{(0)}}\right)
\le
\frac{\kappa}{2},
\end{equation*}
the parameter \(t^{(j)}\) decreases by at least
$
\frac{\kappa}
{4\Theta_1\left(\frac{\theta_\rho(w^{(0,0)};\mu^{(0)})}{\mu^{(0)}}\right)}
$
at each update. Hence, after at most
\begin{equation*}
\O\left(
\frac{
4t^{(0)}
}{\kappa}\Theta_1\left(\frac{\theta_\rho(w^{(0,0)};\mu^{(0)})}{\mu^{(0)}}\right)
\right)
\end{equation*}
updates, we obtain an index \(j\) such that
$
t^{(j)}\Theta_1\left(\frac{\theta_\rho(w^{(0,0)};\mu^{(0)})}{\mu^{(0)}}\right)
\le
\frac{\kappa}{2}.
$
Combining this estimate with \eqref{eq:phase 1 delta} and \eqref{ini-proof-4}, we get
\begin{equation*}
\begin{aligned}
\delta_\rho(w^{(0,j)};\mu^{(0)})
&=
\left\|
(D^2_{ww}\eta_\rho(w^{(0,j)}))^{-1}
\nabla_w\eta_\rho(w^{(0,j)})
\right\|_{\eta_\rho,w^{(0,j)}}                                                   \\
&\le
\delta_{t^{(j)},\rho}(w^{(0,j)})
+
t^{(j)}
\left\|
(D^2_{ww}\eta_\rho(w^{(0,j)}))^{-1}
\nabla_w\eta_\rho(w^{(0,0)})
\right\|_{\eta_\rho,w^{(0,j)}}                                                   \\
&\le
\frac{\kappa}{2}
+
t^{(j)}\Theta_1\left(\frac{\theta_\rho(w^{(0,0)};\mu^{(0)})}{\mu^{(0)}}\right)
\le
\kappa .
\end{aligned}
\end{equation*}
Therefore, Algorithm~\ref{alg:initiation} terminates. The final Newton step computed
from \eqref{SNM-Newton-direction-rho} then gives a point $v^{(0)}\in\mathcal N(\kappa,\mu^{(0)},\rho)$.
Thus Algorithm~\ref{alg:initiation} requires at most
\begin{equation*}
  \O\left(
1+
\frac{
4t^{(0)}
}{\kappa}\Theta_1\left(\dfrac{\theta_\rho(w^{(0,0)};\mu^{(0)})}{\mu^{(0)}}\right)
\right)
\end{equation*}
iterations to obtain a starting point
\(v^{(0)}\in\mathcal N(\kappa,\mu^{(0)},\rho)\).
\end{proof}

Having established the iteration complexity for the first phase of PFSNM, we now turn to analyze the iteration complexity of Algorithm~\ref{alg:main}. As a preliminary step, Theorems~\ref{theorem 7} and~\ref{thm:twice gradient2} quantify the effect of updating the smoothing parameter $\mu$ on both $\nabla_{w} \et$ and $S_{\eta_\rho}(w;\mu)$. For notational simplicity, we omit the arguments $(w;\mu)$ and write, for example, $\nabla_{w} \eta_\rho := \nabla_{w} \eta_\rho(w;\mu)$  whenever first- or second-order derivatives of $\eta_\rho$ with respect to $w=(\hx,s)$ are mentioned. All derivatives with respect to $\mu$ are denoted by a prime.
\begin{theorem}\label{theorem 7}
	For any $\mu>0$, $\rho>0$, any direction $h=(h_{\hx},h_s) \in\hE\times \E$, and any point $w=(\hx,s)\in \hE\times \E$, 
	\begin{equation}
		\left| \langle h, \nabla_{w} \eta_\rho'(w;\mu) \rangle \right|\leq\sqrt{\dfrac{2\nu}{\mu}} \sqrt{S_{\eta_\rho}(w;\mu) [h,h]}.
	\end{equation}
\end{theorem}
\begin{proof}
	Combining Theorem~\ref{thm:differentiability} and Corollary~\ref{coro:differentiability-and-derivatives-of-eta} yields
	$$
	\begin{aligned}
		&\nabla_{w} \eta_\rho'(w;\mu) = \begin{pmatrix} -\B^* y'_\rho \\ z'_\rho \end{pmatrix} = \begin{pmatrix} \B^* \H^{-1}\nabla\phi(z_\rho) \\ -\rho^{-1}\H^{-1}\nabla\phi(z_\rho) \end{pmatrix}.
	\end{aligned}
	$$
	For any $h = (h_{\hx},h_s)\in \hE\times \E$, 
	$$
	 \langle \nabla_{w} \eta_\rho'(w;\mu), h\rangle = \langle \B h_{\hx}, \H^{-1}\nabla\phi(z_\rho)\rangle - \rho^{-1} \langle h_s, \H^{-1}\nabla\phi(z_\rho) \rangle.
	$$
	Let $h_x = \B h_{\hx}$. Then, 
	\begin{equation}\label{eq:thm9-eq1}
     \begingroup 
    \fontsize{9.2pt}{12pt}\selectfont
	\begin{aligned}
		\left|\langle \nabla_{w} \eta'_\rho, h\rangle \right|
        &\leq \left| \langle h_x, \H^{-1}\nabla\phi(z_\rho)\rangle \right| + \rho^{-1}\left| \langle h_s, \H^{-1}\nabla\phi(z_\rho) \rangle \right| \\
		& = \left| \langle h_x, \H^{-\frac{1}{2}}\W^{\frac{1}{2}} \W^{-\frac{1}{2}} \H^{-\frac{1}{2}}\nabla\phi(z_\rho) \rangle \right| + \rho^{-1}\left| \langle h_s, \H^{-\frac{1}{2}} \H^{-\frac{1}{2}}\nabla\phi(z_\rho) \rangle \right| \\
		& \leq \sqrt{\langle h_x, \H^{-1} \W h_x \rangle}\sqrt{ \langle \nabla\phi(z_\rho), \W^{-1} \H^{-1}\nabla\phi(z_\rho) \rangle} \\
        & \qquad \qquad \qquad + \rho^{-1}\sqrt{\langle h_s, \H^{-1} h_s\rangle}\sqrt{ \langle \nabla\phi(z_\rho), \H^{-1}\nabla\phi(z_\rho) \rangle }.
	\end{aligned}
    \endgroup
	\end{equation}
    Since $\H^{-1} = (\I_{\E} + \W)^{-1}$, we have
    $$
     \W^{-1}\H^{-1} \prec \W^{-1}, \quad \H^{-1} \prec \W^{-1},
    $$
    which together with $\W = \frac{\mu}{\rho} D^2 \phi(z_\rho)$ and \eqref{eq:derivatives of phi} implies 
    \begin{equation}\label{eq:thm9-eq2}
        \begin{aligned}
            \sqrt{ \langle \nabla\phi(z_\rho), \W^{-1} \H^{-1}\nabla\phi(z_\rho) \rangle} & \le \sqrt{ \frac{\rho}{\mu} \langle \nabla\phi(z_\rho), (D^2 \phi(z_\rho))^{-1}\nabla\phi(z_\rho) \rangle} = \sqrt{\frac{\rho \nu}{\mu}},\\ 
            \sqrt{ \langle \nabla\phi(z_\rho), \H^{-1}\nabla\phi(z_\rho) \rangle } & \le \sqrt{ \frac{\rho}{\mu} \langle \nabla\phi(z_\rho), (D^2 \phi(z_\rho))^{-1}\nabla\phi(z_\rho) \rangle} = \sqrt{\frac{\rho \nu}{\mu}}.
        \end{aligned}
    \end{equation}
    Combining \eqref{eq:thm9-eq1}, \eqref{eq:thm9-eq2} and Corollary~\ref{coro:differentiability-and-derivatives-of-eta} yields
    $$
       \begin{aligned}
		|\langle \nabla_{w} \eta'_\rho, h\rangle|
		& \le \sqrt{\dfrac{\nu}{\mu}}\left(\sqrt{ \langle h_{\hx},D_{\hx \hx}^2\eta_\rho h_{\hx}\rangle}+\sqrt{-\langle h_s,D_{ss}^2\eta_\rho  h_s\rangle}\right) \\
		& \leq \sqrt{\dfrac{2\nu}{\mu}}\sqrt{S_{\eta_\rho}[h,h]}.
	\end{aligned}
    $$
    This completes the proof.
\end{proof}

\begin{theorem}\label{thm:twice gradient2}
		For any $\mu>0$, $\rho>0$, any direction $h=(h_{\hx},h_s) \in\hE\times \E$, and any point $w=(\hx,s)\in \hE\times \E$, 
		\begin{equation}
			\left| S_{\eta_\rho}'(w;\mu) [h,h]\right|\leq \frac{1+{2\sqrt{\nu}}}{\mu}  S_{\eta_\rho}(w;\mu) [h,h].
		\end{equation}
\end{theorem}
\begin{proof}
Let $h_x = \B h_{\hx}$ and $\h = (\h_x,\h_s) = (\H^{-1} h_x,\H^{-1} h_s)$. Define 
$$ 
\omega(\mu):= S_{\eta_\rho}(w;\mu) [h,h].
$$ 
By Corollary~\ref{coro:differentiability-and-derivatives-of-eta}, 
$$
\omega(\mu)=\omega_x(\mu)+\omega_s(\mu),
$$
where
$
\omega_x(\mu):=\rho \langle h_x, (\I_{\E} - \H^{-1}) h_x\rangle, \, \omega_s(\mu):=\rho^{-1} \langle h_s, \H^{-1} h_s \rangle.
$
Differentiating $\omega_x(\mu)$ with respect to $\mu$ yields
\begin{equation}
    \begin{aligned}\label{eq:xi1(mu)}
        \left| \omega'_x(\mu) \right|&= \left|\langle \H^{-1} h_x, D^2\phi(z_\rho)\H^{-1}h_x \rangle +\mu D^3\phi(z_\rho)[z_\rho', \H^{-1}h_x, \H^{-1} h_x ] \right|\\
        & \le  \left|\frac{\rho}{\mu} \langle h_x, \H^{-1} \W \H^{-1}h_x\rangle \right| + \left|\mu D^3\phi(z_\rho)[z_\rho',\h_x,\h_x] \right|.
    \end{aligned}
\end{equation}
By Corollary~\ref{coro:differentiability-and-derivatives-of-eta} and the inequality $\H^{-1} \W \H^{-1} \prec \H^{-1} \W$, we have 
\begin{equation}\label{eq:xi1(mu)-1}
        \left| \rho \langle h_x, \H^{-1} \W \H^{-1}h_x\rangle \right|  \le  D^2_{\hx\hx}\eta_\rho [h_{\hx},h_{\hx}].
\end{equation}
Since $\phi$ is a standard self-concordant convex function, it follows from \cite[Proposition 9.1.1]{nesterov1994interior} and \eqref{eq:xi1(mu)-1} that 
\begin{equation}\label{eq:xi1(mu)-2}
\begin{aligned}
        \left| \mu D^3\phi(z_\rho)[z_\rho',\h_x,\h_x] \right| & \le 2\mu\sqrt{D^2\phi(z_\rho)[z_\rho',z_\rho']}D^2\phi(z_\rho)[\h_x,\h_x] \\
        & = 2\rho \sqrt{D^2\phi(z_\rho)[z_\rho',z_\rho']}\left< h_x, \H^{-1} \W \H^{-1}h_x\right>\\
        & \le 2 \sqrt{D^2\phi(z_\rho)[z_\rho',z_\rho']} D^2_{\hx\hx}\eta_\rho [h_{\hx},h_{\hx}].
\end{aligned}
\end{equation}
By Theorem~\ref{thm:differentiability} and the inequality $\H^{-1} \prec \W^{-1}$, we have 
\begin{equation}\label{eq:xi1(mu)-3}
    \begin{aligned}
        \sqrt{D^2\phi(z_\rho)[z_\rho',z_\rho']} & = \frac{1}{\rho}\sqrt{ \langle \nabla \phi(z_\rho), \H^{-1} D^2 \phi (z_\rho) \H^{-1} \nabla \phi(z_\rho) \rangle  }\\ 
        & \le \frac{1}{\rho} \sqrt{ \langle \nabla\phi(z_\rho), \W^{-1} D^2 \phi(z_\rho) \W^{-1} \nabla \phi(z_\rho) \rangle }\\
        & = \frac{1}{\mu} \sqrt{ \langle \nabla\phi(z_\rho), ( D^2 \phi(z_\rho) )^{-1} \nabla \phi(z_\rho) \rangle }\\
        & = \frac{\sqrt{\nu}}{\mu}.
    \end{aligned}
\end{equation}
Combining \eqref{eq:xi1(mu)}--\eqref{eq:xi1(mu)-3} yields
\begin{equation}\label{eq:xi11(mu)}
\left| \omega'_x(\mu) \right| \le \frac{1+2\sqrt{\nu}}{\mu}  D^2_{\hx\hx}\eta_\rho [h_{\hx},h_{\hx}].
\end{equation}
Similarly, we obtain
\begin{equation}\label{eq:xi2(mu)}
        \left|\omega'_s(\mu) \right|
        \leq\frac{1+2\sqrt{\nu}}{\mu}\left(-D^2_{ss}\eta_\rho [h_s,h_s]\right).
\end{equation}
Since $\omega(\mu)=\omega_x(\mu)+\omega_s(\mu)$, it follows from the triangle inequality, \eqref{eq:xi11(mu)}, and \eqref{eq:xi2(mu)} that
\begin{equation*}
\begin{aligned}
    |\omega'(\mu)|&\leq \dfrac{1+2\sqrt{\nu}}{\mu}D^2_{\hx\hx}\eta_\rho [h_{\hx},h_{\hx}]-\frac{1+2\sqrt{\nu}}{\mu}D^2_{ss}\eta_\rho [h_s,h_s]\\
    &= \frac{1+{2\sqrt{\nu}}}{\mu}  S_{\eta_\rho} [h,h].
    \end{aligned}
\end{equation*}
This completes the proof.
\end{proof}

The preceding two theorems quantify how the first- and second-order quantities induced by the reduced BAL function vary when the smoothing parameter is changed. These estimates will be used to control the effect of the parameter $\mu$ update in the path-following phase. Before deriving the final complexity bound, we also need an estimate that connects the computable  quantity $\xi_\rho(w;\mu)$ with the primal-dual gap $\theta_\rho(w;\mu)$. For this purpose, we introduce auxiliary displacement measures from the current point to the exact partial minimizer and maximizer
of the reduced minimax problem:
\begin{subequations}\label{definition delta analysis}
\begin{align}
	\tilde{\delta}_{\hx,\rho}(w;\mu) &= \sqrt{\dfrac{1}{\mu} \langle \widetilde{\Delta \hx}, D^2_{\hx \hx}\et \widetilde{\Delta \hx}\rangle},\label{definition delta x2}\\
	\tilde{\delta}_{s,\rho}(w;\mu) &= \sqrt{-\dfrac{1}{\mu}\langle \widetilde{\Delta s},D^2_{ss}\et\widetilde{\Delta s}\rangle},\label{definition delta s2}\\
	\tilde{\delta_\rho}(w;\mu) &= \sqrt{(\tilde{\delta}_{\hx,\rho}(w;\mu))^2+(\tilde{\delta}_{s,\rho}(w;\mu))^2} = \| \widetilde{\Delta w}\Vert_{\eta_\rho,w},\label{definition delta 2}
\end{align}
\end{subequations}
where $\widetilde{\Delta \hx}=\hx-\hx_\rho (s,\mu)$, $\widetilde{\Delta s}=s-s_\rho(\hx,\mu)$, and $\widetilde{\Delta w}=( \widetilde{\Delta \hx},\widetilde{\Delta s} )$. These quantities are used only in the analysis
and are not required in the implementation of the algorithm.

\begin{theorem}\label{thm:theta-estimate}

For any $\mu>0$, $\rho>0$, and any point $w=(\hx,s)\in \hE\times \E$,
suppose that $\kappa=0.1$ and 
    $$
    \xi_\rho(w;\mu)\le \kappa.
    $$ 
Then the primal-dual gap function satisfies
    \begin{equation*}
		\theta_\rho (w;\mu) \leq \kappa\mu.
	\end{equation*}
\end{theorem}
\begin{proof}
	   To quantify the gaps between the current value $\eta_\rho$ and the primal optimal value, define
       \begin{equation*}
       \begin{aligned}
           \theta_{\hx,\rho}(\mu)& := \eta_\rho(\hx,{s};\mu)-\eta_\rho(\hx_\rho (s,\mu),{s};\mu),\\
           \theta_{s,\rho}(\mu) &:= \eta_\rho(\hx,{s}_\rho(\hx,\mu);\mu)-\eta_\rho(\hx,{s};\mu).
       \end{aligned}
       \end{equation*}
     If $\xi_\rho(w;\mu) \le \kappa$, then by Theorem~\ref{thm:properties-for-self-concordant-convex-concave-function}(iv),
    $$ 
        \max\left\{ \tilde{\delta}_{\hx,\rho}(w;\mu),   \tilde{\delta}_{s,\rho}(w;\mu)\right\}\le  2\kappa.
    $$
    Consequently, we have
	$$
	\begin{aligned}
	\theta_{s,\rho}(\mu)&=\eta_\rho(\hx,{s}(\hx,\mu);\mu)-\eta_\rho(\hx,{s};\mu)\\
	&=-\int_{0}^{1} \langle \widetilde{\Delta s}, \nabla_s\eta_\rho(\hx,s_\rho(\hx,\mu)+\tau\widetilde{\Delta s};\mu)\rangle \; d\tau\\
	&=-\int_{0}^{1}\int_{0}^{\tau} \langle \widetilde{\Delta s}, D^2_{ss}\eta_\rho(\hx,s_\rho (\hx,\mu)+t\widetilde{\Delta s};\mu)\widetilde{\Delta s}\rangle \; dt d\tau\\
	&{\leq\int_{0}^{1}\int_{0}^{\tau} \dfrac{\mu\widetilde{\delta}_{s,\rho}^2}{(1-\widetilde{\delta}_{s,\rho}+t \widetilde{\delta}_{s,\rho})^2}\; dt d\tau}\\
	&=\left(\frac{\tilde{\delta}_{s,\rho}}{1-\tilde{\delta}_{s,\rho}}+\ln(1-\tilde{\delta}_{s,\rho})\right){\mu}\\
	&\leq \frac{\kappa }{2}\mu,
	\end{aligned}
	$$
	where the first inequality follows from Theorem~\ref{theorem 1} and Proposition~\ref{proposition 2}. By the same argument, we conclude that $\theta_{\hx,\rho}(\mu) \leq \frac{\kappa }{2}\mu$. 
	 Therefore, 
	 $$
	 \theta_\rho(w;\mu)=\theta_{s,\rho}(\mu)+\theta_{\hx,\rho}(\mu)\leq \kappa \mu .
	 $$
     This completes the proof.
\end{proof}

\begin{remark}
 By \eqref{error-between-optimal-value-and-eta} and Theorem~\ref{thm:theta-estimate}, if $\xi_\rho(w;\mu)\le\kappa$, then
 $$
   | \eta_\rho(w;\mu)-{\rm val} (\mathrm{P}_\mu) | \le \kappa\mu.
 $$
  Therefore, the PFSNM can be viewed as a relaxation of the IPM. The interpolation between the subproblem objectives of the two methods is controlled by the parameter $\mu$. 
\end{remark}

The following lemma relates the merit function $\xi_\rho(w;\mu)$ to the quantities $\nabla_{w} \et$ and $S_{\eta_\rho}(w;\mu)$, which is crucial for the subsequent complexity analysis.

\begin{lemma}\label{lemma:delta-expression}
For any $\mu>0$, $\rho> 0$, and any point $w=(\hx,s)\in \hE\times \E$, 
\begin{equation}
    \xi_\rho (w;\mu) = \max \limits_{0\neq h\in 
    \hE\times \E} \frac{ \left| \langle \nabla_{w} \et, h \rangle \right|}{\sqrt{ \mu  S_{\eta_\rho} (w;\mu)[h,h]}}.
\end{equation}
\end{lemma}
\begin{proof}
By Theorem~\ref{thm:self-concordant-of-eta}, $\eta_\rho(\cdot,s;\mu)$ is nondegenerate $\mu$-self-concordant on $\hE$ for every $s\in \E$, and $-\eta_\rho(\hx,\cdot;\mu)$ is nondegenerate $\mu$-self-concordant on $\E$ for every $\hx \in \hE$. It follows from \cite[Proposition 2.2.1]{nesterov1994interior} that for any nonzero direction $ h=(h_{\hx}, h_s) \in \hE\times \E$, 
\begin{equation*}
\begin{aligned}
         \sqrt{\mu}\xi_{\hx,\rho} (w;\mu) \ge  \frac{\left| \langle \nabla_{\hx} \et, h_{\hx}  \rangle \right|}{ \sqrt{ \langle h_{\hx}, D^2_{\hx \hx} \et h_{\hx} \rangle } },\\
         \sqrt{\mu}\xi_{s,\rho} (w;\mu) \ge  \frac{\left|\langle \nabla_{s} \et, h_s \rangle \right|}{ \sqrt{ -\langle h_s,D^2_{ss} \et h_s \rangle} }.
\end{aligned}
\end{equation*}
Consequently, we have
\begin{equation*}
    \begin{aligned}
         & \sqrt{\mu}\xi_\rho \sqrt{\langle h_{\hx}, D^2_{\hx \hx} \eta_\rho h_{\hx} \rangle  -\langle h_s,D^2_{ss} \eta_\rho h_s \rangle } \\ 
         = &\ \sqrt{\mu} \sqrt{\xi_{\hx,\rho}^2 + \xi_{s,\rho}^2}\sqrt{\langle h_{\hx}, D^2_{\hx \hx} \eta_\rho h_{\hx} \rangle  -\langle h_s,D^2_{ss} \eta_\rho h_s \rangle }\\
         \ge &\ \sqrt{\mu} \left( \xi_{\hx,\rho} \sqrt{\langle h_{\hx}, D^2_{\hx \hx} \eta_\rho h_{\hx} \rangle} + \xi_{s,\rho} \sqrt{-\langle h_s,D^2_{ss} \eta_\rho h_s \rangle}     \right)\\
         \ge &\ \left| \langle \nabla_{\hx} \eta_\rho, h_{\hx}  \rangle \right| +  \left|\langle \nabla_{s} \eta_\rho, h_s \rangle \right|\\
         \ge &\ \left| \langle \nabla_{w} \eta_\rho, h \rangle   \right|,
    \end{aligned}
\end{equation*}
where the first inequality follows from the Cauchy--Schwarz inequality. Thus, 
\begin{equation}\label{lemma:delta-inequality}
    \xi_\rho(w;\mu) \ge \frac{ | \langle \nabla_{w} \et, h \rangle |}{\sqrt{ \mu  S_{\eta_\rho} (w;\mu)[h,h]}}, \quad \forall\, h\in \hE\times \E,\, h\neq0.
\end{equation}
Choosing 
\begin{equation*}
    h = \left( (D_{\hx \hx}^2 \eta_\rho)^{-1} \nabla_{\hx} \eta_\rho, -(D_{s s}^2 \eta_\rho)^{-1} \nabla_{s} \eta_\rho  \right),
\end{equation*}
inequality~\eqref{lemma:delta-inequality} holds with equality, which completes the proof.
\end{proof}

The preceding estimates quantify the sensitivity of the reduced BAL function with respect to the smoothing parameter. The following lemma is the key step in the path-following analysis: it shows that, if the current point is in the neighborhood associated with $\mu$, then the same point remains sufficiently close to the new central path
after replacing $\mu$ by $\sigma\mu$.
\begin{lemma}\label{lemma:effect of mu}
Let $\rho>0$, $\mu>0$, $\kappa=0.1$, and $\mu^+=\sigma\mu$ with 
\[\label{eq:definition of sigma}
\sigma= 1-\dfrac{\ln(\gamma)}{2\sqrt{\nu}+\ln(\gamma)}, \; where \;\gamma:=\dfrac{2\kappa+\sqrt{2}}{\kappa+\sqrt{2}}.
\]
If $\xi_\rho(w;\mu)\le \kappa$, then
\[
\xi_\rho(w;\mu^+)\le 2\kappa .
\]
\end{lemma}
\begin{proof}
For any nonzero $h\in \hE\times \E$,  define the function $\psi: \R_{++} \times \hE \times \E \to \R$ as
\begin{equation*}
    \psi(\tau,h): = \frac{ \langle \nabla_{w} \eta_\rho(w;\tau), h\rangle^2}{ \tau S_{\eta_\rho}(w;\tau) [h,h]}.
\end{equation*}
Differentiating $\psi(\tau,h)$ with respect to $\tau$ gives
\begin{equation*}
    \begin{aligned}
    \psi' (\tau,h) & =  \frac{2 \langle \nabla_{w} \eta_\rho(w;\tau), h\rangle \cdot \langle \nabla_{w} \eta_\rho'(w;\tau), h\rangle}{\tau S_{\eta_\rho}(w;\tau) [h,h]} - \frac{\langle \nabla_{w} \eta_\rho(w;\tau), h\rangle^2 }{ \tau^2 S_{\eta_\rho}(w;\tau) [h,h]}\\
    & \qquad \qquad  -
    \frac{ \langle \nabla_{w} \eta_\rho(w;\tau), h\rangle^2   S'_{\eta_\rho}(w;\tau) [h,h] }{ \tau ( S_{\eta_\rho}(w;\tau) [h,h])^2}.
    \end{aligned}
\end{equation*}
Let $\tau \in [ \mu^+, \mu ]$. It follows from Theorems~\ref{theorem 7}--\ref{thm:twice gradient2} that
\begin{equation*}
    \begin{aligned}
         |\psi' (\tau,h)| &\le \sqrt{ \frac{8\nu}{\tau^2} } \sqrt{ \psi(\tau,h)} + \frac{2+2\sqrt{\nu}}{\tau} \psi (\tau,h)\\
         & \le \sqrt{ \frac{8\nu}{(\mu^+)^2} } \sqrt{ \psi(\tau,h)} + \frac{2+2\sqrt{\nu}}{\mu^{+}} \psi (\tau,h).
    \end{aligned}
\end{equation*}
Define $c_1:= \frac{1+\sqrt{\nu}}{ \mu^+}$ and $\Psi (\tau,h) := e^{c_1 \tau} \sqrt{\psi(\tau,h)}$. We have
\begin{equation}\label{Psi-derivative}
    -\Psi'(\tau,h) \le e^{c_1 \tau} \frac{\sqrt{2\nu}}{\mu^{+}}.
\end{equation}
Integrating both sides of \eqref{Psi-derivative} over $[\mu^+,\mu]$ yields
\begin{equation*}
    \Psi( \mu^+,h) - \Psi( \mu,h) \le \frac{ \sqrt{2\nu}}{\mu^{+}} \frac{1}{c_1} \left(e^{c_1 \mu} -  e^{c_1 \mu^{+}}    \right).
\end{equation*}
This implies that for any nonzero direction $h\in \hE\times \E$, 
\begin{equation}\label{eq:merit function}
\begin{aligned}
     &\sqrt{\psi(\mu^{+},h)}\\
     \le &\, e^{c_1 (\mu-\mu^{+})} \sqrt{\psi (\mu,h)} +\frac{\sqrt{2\nu}}{\mu^{+}}\frac{1}{c_1} \left( e^{ c_1 (\mu-\mu^{+})}- 1   \right)\\
      \le &\, e^{c_1 (\mu-\mu^{+})} \xi_\rho(w;\mu) +\frac{\sqrt{2\nu}}{\mu^{+}} \frac{1}{c_1} \left( e^{ c_1 (\mu-\mu^{+})}- 1   \right),
\end{aligned}
\end{equation}
where the last inequality follows from Lemma~\ref{lemma:delta-expression}. Since \eqref{eq:merit function} holds for every nonzero direction $h\in \hE\times \E$, we conclude by Lemma~\ref{lemma:delta-expression} again that
\begin{equation}
\begin{aligned}\label{eq:aux eq1}
       \xi_\rho(w;\mu^{+}) &\le e^{c_1 (\mu-\mu^{+})} \xi_\rho(w;\mu) 
   +\frac{\sqrt{2\nu}}{\mu^{+}} \frac{1}{c_1}\left( e^{ c_1 (\mu-\mu^{+})}- 1   \right).
\end{aligned}
\end{equation}
Recall $\mu^{+} = \sigma \mu$. Let
\begin{equation}\label{eq:aux eq2}
    \left\{\begin{array}{ll}
         & \alpha_1 (\sigma):= e^{c_1 (\mu-\mu^{+})}  = e^{ (1 + \sqrt{\nu}) (\frac{1}{\sigma} -1) }, \\
         & \alpha_2:=\frac{\sqrt{2\nu}}{\mu^{+}} \frac{1}{c_1} = \frac{ \sqrt{2\nu }}{1 +\sqrt{\nu}}.
    \end{array}\right.
\end{equation}
It can be verified that
\begin{equation*}
    \alpha_1(\sigma) =e^{(1+\sqrt{\nu})\cdot\frac{\ln(\gamma)}{2\sqrt{\nu}}} \le  \gamma \text{ \, and\,  } \alpha_2 \le \sqrt{2},
\end{equation*}
whenever $\sigma= 1-\dfrac{\ln(\gamma)}{2\sqrt{\nu}+\ln(\gamma)}$.

Combining this estimate with \eqref{eq:aux eq1} and \eqref{eq:aux eq2}, we get
\begin{equation*}
    \xi_\rho(w;\mu^{+}) \le \alpha_1 (\sigma) \kappa + \alpha_2( \alpha_1 (\sigma) - 1)\le 2\kappa,
\end{equation*}
which completes the proof.
\end{proof}

Building on these estimates, we establish the main result of the paper: Algorithm~\ref{alg:main} admits a polynomial iteration complexity of order $\O(\sqrt{\nu}\ln(1/\varepsilon))$, matching the best--known complexity of classical short-step IPMs.

\begin{theorem}\label{thm:complexity of phase 2}
Let $\rho> 0$, $\kappa=0.1$, and choose $\sigma$ as defined in \eqref{eq:definition of sigma}. Suppose that the initial point $v^{(0)}$
is generated by Algorithm~\ref{alg:initiation}. Then Algorithm~\ref{alg:main} requires at most $\O(\sqrt{\nu}\ln(\mu^{(0)}/\varepsilon))$ iterations to attain the desired accuracy~$\varepsilon$.
\end{theorem}
\begin{proof}

We first show the iterates remain in the neighborhood $\mathcal{N}(\kappa, \mu^{(k)},\rho)$ for any $k\ge 0$. The claim holds for $k=0$ by the initialization. Suppose that
$
v^{(k)}\in \mathcal N(\kappa,\mu^{(k)},\rho),
$
which implies that 
$$
\xi_\rho(w^{(k)};\mu^{(k)})\le \kappa.
$$
If $\mu^{(k)}\le\varepsilon$, the algorithm terminates. Otherwise, the smoothing parameter is reduced according to
$
\mu^{(k+1)}=\sigma\mu^{(k)}.
$
Applying Lemma~\ref{lemma:effect of mu} with
$
w=w^{(k)},\, \mu=\mu^{(k)},\, \mu^+=\mu^{(k+1)},
$
we obtain
\begin{equation*}
\xi_\rho(w^{(k)};\mu^{(k+1)})\le 2\kappa.
\end{equation*}
Since $\kappa=0.1$, we have $2\kappa=0.2<2-\sqrt{3}$. 
Let $w^{(k+1)}$ denote the point obtained after one full Newton step. By Theorem~\ref{thm:properties-for-self-concordant-convex-concave-function}(iii)
$$
\xi_\rho(w^{(k+1)};\mu^{(k+1)}) \le \frac{\xi_\rho(w^{(k)};\mu^{(k+1)}) }{2}\leq\kappa.
$$
This, combined with \eqref{eq:initial=0}, implies that $ v^{(k+1)} \in \mathcal{N}(\kappa,\mu^{(k+1)},\rho)$.
By induction, we conclude that $v^{(k)}$ remains in the neighborhood $\mathcal{N}(\kappa,\mu^{(k)},\rho)$ for any $k\ge 0$.

Now we bound the number of Newton steps. Since $\mu^{(k)} = \sigma^k \mu^{(0)}$, the stopping criterion $\mu^{(k)}\le \varepsilon$ of Algorithm~\ref{alg:main} is satisfied whenever
$$
k \ge \frac{\ln({\mu^{(0)}}/{\varepsilon})}{ \ln(1/\sigma)}.
$$
Consequently, Algorithm~\ref{alg:main} requires at most
$$
 -\ln\!\left(\frac{\mu^{(0)}}{\varepsilon}\right) /\ln(\sigma) + 1
= \mathcal{O}\!\left(\sqrt{\nu}\ln\!\left(\frac{\mu^{(0)}}{\varepsilon}\right)\right)
$$
iterations to obtain an approximate KKT triple.
\end{proof}

\begin{remark}
    Combining Theorems \ref{initialization-complexity-analysis} and \ref{thm:complexity of phase 2}, we obtain that the first phase of PFSNM requires at most $     \O\left(
1+
\frac{
4t^{(0)}
}{\kappa}\Theta_1(\frac{\theta_\rho(w^{(0,0)};\mu^{(0)})}{\mu^{(0)}})
\right)$ iterations, whereas the second phase of PFSNM admits an iteration complexity of $\mathcal{O}\left(\sqrt{\nu} \ln\left(\mu^{(0)}/{\varepsilon}\right)\right)$. Therefore, the overall iteration complexity is $$    \O\left(
1+
\frac{
4t^{(0)}
}{\kappa}\Theta_1\left(\dfrac{\theta_\rho(w^{(0,0)};\mu^{(0)})}{\mu^{(0)}}\right)
\right)+\mathcal{O}\left(\sqrt{\nu} \ln\left(\frac{\mu^{(0)}}{\varepsilon}\right)\right).$$
Since the first term is independent of $\varepsilon$, the above bound can be written more compactly as $\mathcal{O}\left(\sqrt{\nu} \ln\left(1/{\varepsilon}\right)\right)$.
\end{remark}

\section{Computational results}\label{sec:computational-results}

To validate the effectiveness of the proposed method (PFSNM), this section reports numerical results on three benchmarks and compares PFSNM with several widely used conic programming solvers, including SDPT3 \cite{tutuncu2003solving}, SeDuMi \cite{sturm1999using}, ECOS \cite{domahidi2013ecos}, and Clarabel \cite{goulart2024clarabel}. The test instances consist of linear programs from the NETLIB collection\footnote{\url{https://netlib.org/lp/data/}}, convex quadratic programs (QP) from the Maros--M{\'e}sz{\'a}ros collection\footnote{\url{https://www.doc.ic.ac.uk/~im/}}, and second-order cone programs arising from square-root Lasso formulations constructed using matrices from the SuiteSparse Matrix Collection\footnote{\url{https://sparse.tamu.edu/}}.  All computational results are obtained on a Windows~10 personal computer equipped with an Intel i5-8300H processor (4 cores, 8 threads, 2.3~GHz) and 16~GB of RAM. The proposed method is implemented in C.

We evaluate solver performance using performance profiles \cite{dolan2002benchmarking} and the shifted geometric mean (SGM)\footnote{\url{https://plato.asu.edu/ftp/shgeom.html}}. Let $\mathcal{P}$ denote the benchmark set and $\mathcal{S}$ denote the set of solvers. For each problem $p\in\mathcal{P}$ and solver $s\in\mathcal{S}$, let $t_{p,s}$ be the runtime, and define the performance ratio
\begin{equation*}
r_{p,s}=\frac{t_{p,s}}{\min_{s'\in\mathcal S}t_{p,s'}}\in[1,\infty],
\end{equation*}
where $r_{p,s}=\infty$ if solver $s$ fails to solve problem $p$ within the time limit of 1000 seconds. The performance profile is given by
\begin{equation*}
\rho_s(\tau)=\frac{1}{|\mathcal P|}\Big|\big\{p\in\mathcal P:\ r_{p,s}\le \tau\big\}\Big|,\qquad \tau\ge 1,
\end{equation*}
which measures the fraction of instances for which $s$ is within a factor $\tau$ of the best solver. The value at $\tau=1$ reflects the frequency with which solver $s$ is the fastest, while the limiting value as $\tau$ grows measures its empirical success rate. In addition, we summarize the overall performance via the shifted geometric mean (with offset $=1$)  
\begin{equation*}
\mathrm{SGM}_s=\exp\!\left(\sum_{p\in\mathcal P}\frac{1}{{|\mathcal P|}}\ln\left({\max\{1,t_{p,s}+\text{offset}\}}\right)\right)-\text{offset},
\end{equation*}
so that smaller values indicate better aggregate efficiency while being relatively insensitive to a small number of difficult instances.

\subsection{Linear programs}\label{sec:computational-results.1}
We first test the solvers on linear programs from the widely used NETLIB collection. The results are summarized in Table~\ref{tab:LP} and Fig.~\ref{fig:performance_LP}.

Table~\ref{tab:LP} reports the solved ratios together with the shifted geometric means. The main observation is that PFSNM achieves a favorable balance between robustness and efficiency on this benchmark. In particular, it achieves the best aggregate runtime behavior as measured by the SGM without compromising reliability, whereas competing solvers exhibit either a lower success rate or a noticeably larger SGM. This indicates that the advantage of PFSNM on NETLIB is consistent across instances, reflecting a favorable overall balance between convergence robustness and computational cost.

\begin{table}[H]
\centering
\caption{Solved ratios and SGMs of PFSNM, SDPT3, SeDuMi, ECOS, and Clarabel on the NETLIB collection.}
\label{tab:LP}
\begin{tabular}{cccccc}
\toprule
Solver & PFSNM & SDPT3 & SeDuMi & ECOS & Clarabel \\
\midrule
Solved ratio &  \textbf{100\%} &  73.47\% & 93.88\% &  95.92\% & 97.96\%\\
\hline
 SGM & \textbf{1.0000} & 25.0299 & 3.5985 & 2.7778 & 2.0520\\
\bottomrule
\end{tabular}
\end{table}

Figure~\ref{fig:performance_LP} provides a more intuitive perspective through performance profiles. The curve of PFSNM stays above the competing solvers over essentially the entire range of $\tau$. This behavior indicates that PFSNM is frequently the fastest among the five solvers on a large fraction of the benchmark set. Furthermore, $\rho_s(\tau)$ for PFSNM approaches $1$ as $\tau$ increases, which means that it successfully solves all LP instances under the imposed limits. In contrast, the profiles of the other solvers level off below $1$, and their slower rise for small $\tau$ indicates weaker performance on instances with comparable runtimes.
\begin{figure}[H]
    \centering
    \includegraphics[width=0.65\linewidth]{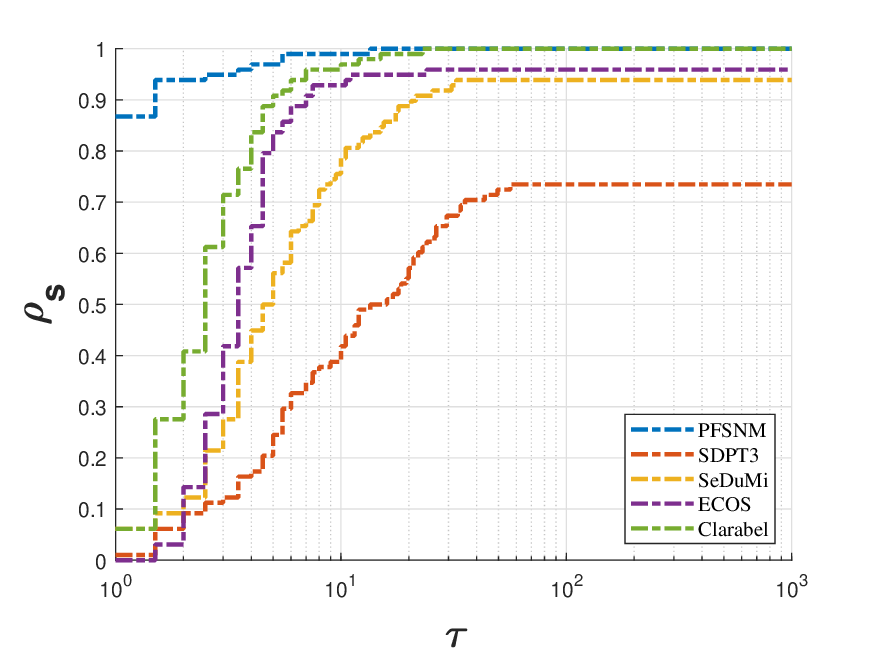}
    \caption{Performance profiles of PFSNM, SDPT3, SeDuMi, ECOS, and Clarabel on the NETLIB collection}
    \label{fig:performance_LP}
\end{figure}

\subsection{Quadratic programs}\label{sec:computational-results.2}
We next consider convex quadratic programs from the Maros--M{\'e}sz{\'a}ros collection, a standard benchmark for QP problems. The results are summarized in Table~\ref{tab:QP-maros} and Fig.~\ref{fig:performance_QP_maros}.

Table~\ref{tab:QP-maros} suggests that PFSNM is competitive in terms of aggregate efficiency, although it is not the best-performing solver overall. In particular, its shifted geometric mean is the second-best among the tested solvers (about $2.66$), whereas Clarabel attains the best value (normalized to $1.00$). At the same time, the solved ratios indicate that Clarabel succeeds on a larger fraction of the QP instances (about $91.3\%$), while PFSNM solves a smaller but still substantial fraction (about $82.6\%$). Overall, the table indicates that the main difference between PFSNM and the best solver on this benchmark lies in robustness on a subset of instances.
\begin{table}[H]
\centering
\caption{Solved ratios and SGMs of PFSNM, SDPT3, SeDuMi, ECOS, and Clarabel on the Maros--M{\'e}sz{\'a}ros collection.}
\label{tab:QP-maros}
\begin{tabular}{cccccc}
\toprule
Solver & PFSNM & SDPT3 & SeDuMi & ECOS & Clarabel \\
\midrule
Solved ratio &  82.61\% &  77.54\% & 83.33\% &  72.46\% & \textbf{91.30\%}\\
\hline
 SGM & 2.6598 & 8.1669 & 5.8094 & 7.7222 & \textbf{1.0000}\\
\bottomrule
\end{tabular}
\end{table}
Figure~\ref{fig:performance_QP_maros} provides a consistent view. The performance profile of PFSNM rises quickly for small values of $\tau$, which indicates that when PFSNM succeeds, its runtimes are often close to those of the best solver on those instances. However, the limiting value of its profile remains below that of the most reliable solver, reflecting the gap in solved ratios reported in Table~\ref{tab:QP-maros}. Overall, the QP results show that PFSNM can be fast on the instances it solves. Improving robustness on the more difficult subset of Maros--M{\'e}sz{\'a}ros problems would further enhance its performance on this benchmark.
\begin{figure}[H]
    \centering
    \includegraphics[width=0.65\linewidth]{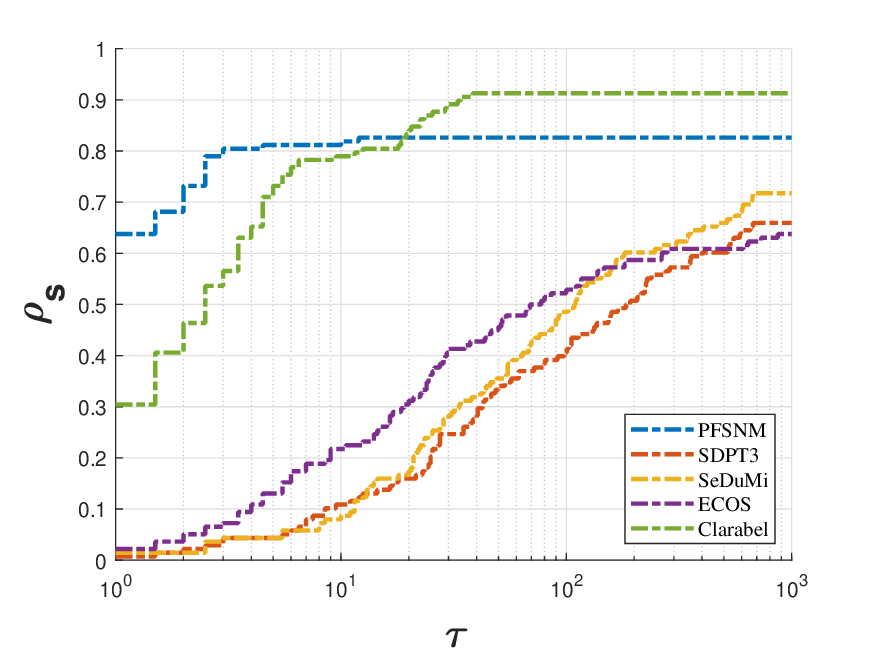}
    \caption{Performance profiles of PFSNM, SDPT3, SeDuMi, ECOS, and Clarabel on the Maros--M{\'e}sz{\'a}ros collection}
    \label{fig:performance_QP_maros}
\end{figure}
\subsection{Second-order cone programs}\label{sec:computational-results.3}
Finally, we evaluate the solvers on a family of SOCP instances constructed from Lasso formulations. The data matrices are drawn from the SuiteSparse Matrix Collection, and each matrix is used to build a square-root Lasso problem \cite{belloni2011square,liang2021inexact} of the form
\begin{equation*}
    \min\limits_{y \in \R^{n}} \left\{ \Vert D y - b \Vert_2 + \varrho \Vert y \Vert_1      \right\},
\end{equation*}
where $D\in \R^{d\times n}$ is a matrix from the SuiteSparse Matrix Collection, $b \in \R^d$ is a given vector, and $\varrho$ is a penalty parameter. This problem is equivalent to the following SOCP reformulation:
\begin{equation*}
    \min \left\{ t+\varrho \sum\limits_{i=1}^n (y^+_i + y^-_i) \, \Big{|}\, Dy^+ - Dy^- -r =b,\, (t,r)\in \mathbb{Q}^{d+1},\, y^+,y^- \in \R^n_+  \right\}.
\end{equation*}
Following \cite{goulart2024clarabel}, we choose $\varrho = \Vert D^\top b \Vert_{\infty}$. The vector $b$ is set to the all-ones vector.
The results are reported in Table~\ref{tab:SOCP-suitesparse} and Fig.~\ref{fig:performance_profile_SOCP_suitesparse}.

Table~\ref{tab:SOCP-suitesparse} shows that all solvers solve all the instances, so the comparison is driven by efficiency rather than robustness. In this setting, PFSNM attains the smallest SGM (normalized to $1.0000$). The closest competitor is Clarabel, with an SGM only slightly larger (about $1.0818$), whereas the remaining solvers have clearly larger SGM values. Hence, the table indicates that PFSNM provides the best aggregate runtime performance on these Lasso-type SOCPs, with particularly close performance relative to Clarabel.
\begin{table}[H]
\centering
\caption{Solved ratios and SGMs of PFSNM, SDPT3, SeDuMi, ECOS, and Clarabel on SOCP problems constructed from SuiteSparse matrices.}
\label{tab:SOCP-suitesparse}
\begin{tabular}{cccccc}
\toprule
Solver & PFSNM & SDPT3 & SeDuMi & ECOS & Clarabel \\
\midrule
Solved ratio &  \textbf{100\%} &  \textbf{100\%} & \textbf{100\%} &  \textbf{100\%} & \textbf{100\%}\\
\hline
 SGM & \textbf{1.0000} & 3.5750 & 2.5645 & 2.3100 & 1.0818\\
\bottomrule
\end{tabular}
\end{table}

Figure~\ref{fig:performance_profile_SOCP_suitesparse} is consistent with the table-based summary. Since all solvers succeed, the key difference lies in how quickly each profile rises near $\tau=1$. The PFSNM curve increases the fastest and stays close to the best observed curve over a wide range of $\tau$, meaning that it achieves the best or near-best runtime on a large fraction of instances. Combined with the SGM results, these experiments indicate that PFSNM offers strong and reliable performance on SOCP problems constructed from SuiteSparse matrices. 
\begin{figure}[H]
    \centering
    \includegraphics[width=0.65\linewidth]{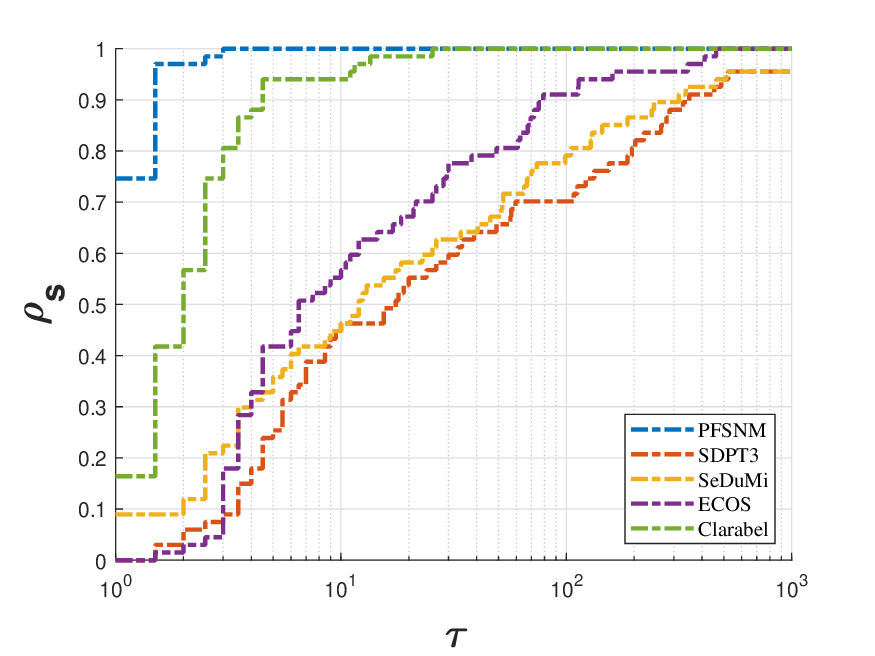}
    \caption{Performance profiles of PFSNM, SDPT3, SeDuMi, ECOS, and Clarabel on SOCP problems constructed from SuiteSparse matrices}
    \label{fig:performance_profile_SOCP_suitesparse}
\end{figure}

\section{Conclusion}\label{sec7}

We have proposed a path-following smoothing Newton method for symmetric cone programming based on a reduced BAL function. The associated parameterized smooth system has been shown to be equivalent to the first-order optimality conditions of a structured minimax problem. This characterization makes it possible to analyze the proposed method within a self-concordant convex-concave framework. We have proved that the reduced BAL function is $\mu$-self-concordant convex-concave and that the resulting method attains a worst-case iteration complexity of $O(\sqrt{\nu}\ln(1/\varepsilon))$. This iteration complexity matches the best-known short-step bound for IPMs on symmetric cones. Moreover, the reduced BAL function also yields Newton systems with an explicit Schur complement, which can be exploited to reduce system-formation costs for important classes of symmetric cones. Numerical results indicate that the method is competitive on standard conic benchmarks.
\appendix

\section{Auxiliary proofs}\label{secA1}
\subsection{Proof of Proposition \ref{proposition 2}}\label{proof:Prop 1}

\begin{proof}
For any point $w=(\hx,s)\in \hE\times  \E $ and any direction $h_{\hx}\in\hE $, define $h = (h_{\hx}, 0)\in\hE\times \E$ and let $\varrho(t) = D^2 f(w + th)[h,h]=D^2_{\hx \hx} f(w+th)[h_{\hx},h_{\hx}]$.
By the $\alpha$-self-concordant convex-concave property of $f$, we have
$$
\begin{aligned}
\varrho'(t) &= D^3_{\hx\hx\hx} f(w+th)[h_{\hx},h_{\hx},h_{\hx}]\\
& = D^3 f(w + th)[h,h,h] \\
& \leq \frac{2}{\alpha^{1/2}} \left( S_f(w+ th)[h, h] \right)^{3/2}\\ &=\frac{2}{\alpha^{1/2}}  \left( D^2_{\hx\hx} f(w+ th)[h_{\hx}, h_{\hx}] \right)^{3/2}.
\end{aligned}
$$
Taking $t=0$ yields
$$
| D^3_{\hx\hx\hx} f(w)[h_{\hx},h_{\hx},h_{\hx}] |\leq \frac{2}{\alpha^{1/2}} \left( D^2_{\hx\hx} f(w)[h_{\hx}, h_{\hx}] \right)^{3/2}.
$$
Moreover, for any $w\in \hE \times \E$ and any nonzero direction $h_{\hx} \in \hE$, 
$$
    D^2_{\hx \hx} f(w)[h_{\hx},h_{\hx}] > 0.
$$
This implies that $f(\cdot,s)$ is nondegenerate $\alpha$-self-concordant on $\hE$ for every $s\in \E$. An analogous argument shows that $-f(\hx,\cdot)$ is nondegenerate $\alpha$-self-concordant on $\E$ for every $\hx\in \hE$.

The conclusion in $(ii)$ follows directly from \cite[Proposition 9.1.1]{nesterov1994interior}. 
\end{proof}
\subsection{Proof of Theorem \ref{thm:properties-for-self-concordant-convex-concave-function}(iv)}\label{proof:Thm 3}

\begin{proof}

Define
	$$
    \begin{aligned}
	\xi_{\hx}(w):=\Big(\tfrac{1}{\alpha}\big\langle \nabla_{\hx} f(w),\big(D^2_{\hx\hx}f(w)\big)^{-1}\nabla_{\hx} f(w)\big\rangle\Big)^{1/2},\\
	\xi_s(w):=\Big(\tfrac{1}{\alpha}\big\langle \nabla_s f(w),\big(-D^2_{ss}f(w)\big)^{-1}\nabla_s f(w)\big\rangle\Big)^{1/2}.
    \end{aligned}
	$$
	By definition,
    $\xi(w)^2=\xi_{\hx}(w)^2+\xi_s(w)^2$.
Let $\varrho_s(\tilde x)=f(\tilde x,s)$ and $\hx(s) = \arg\min_{\tilde x} \varrho_s (\tilde x)$. 
	Define $d:=\hx-\hx(s)$. Recall that
	$$
    \begin{aligned}
	\tilde\delta_{\hx}(w)
	=\Big(\tfrac{1}{\alpha}\langle d,D^2_{\hx\hx}f(w)\,d\rangle\Big)^{1/2}.
    \end{aligned}
	$$
    By Proposition~\ref{proposition 2}, $\varrho_s(\cdot)$ is $\alpha$-self-concordant convex on $\hE$. Consequently, it follows from \cite[Eq.~(2.2.31)]{nesterov1994interior} that
    $$
        \tilde{\delta}_{\hx} (w) \le 1-(1-3\xi_{\hx} (w))^{1/3} \le 1 - (1-3\xi(w))^{1/3}.
    $$
    Similarly, one has 
    $$ 
        \tilde{\delta}_s (w) \le 1-(1-3\xi_s (w))^{1/3} \le 1 - (1-3\xi(w))^{1/3}.
    $$
    Consequently 
    $$ 
        \max\{ \tilde{\delta}_{\hx} (w), \tilde{\delta}_s (w)  \} \le 1 - (1-3\xi(w))^{1/3}.
    $$
    Furthermore, if $\xi(w) \le 0.1$, then 
    $$
    \max\{ \tilde{\delta}_{\hx} (w), \tilde{\delta}_s (w)  \} \le 1-0.7^{1/3} < 0.2,
    $$
    which completes the proof.
\end{proof}

\bibliographystyle{spmpsci}
\bibliography{references}


\end{document}